\documentclass[12pt,a4paper]{article}
\usepackage{amsmath,enumerate,bm,bbm,color}
\usepackage{amsthm}
\usepackage{amssymb}
\usepackage{amsbsy}
\usepackage{amsfonts}
\usepackage{amstext}
\usepackage{amscd}
\numberwithin{equation}{section}
\theoremstyle{plain}
\newtheorem{thm}{Theorem}[section]
\newtheorem{prop}[thm]{Proposition}
\newtheorem{cor}[thm]{Corollary}
\newtheorem{lem}[thm]{Lemma}
\theoremstyle{definition}
\newtheorem{exa}[thm]{Example}
\newtheorem{conj}[thm]{Conjecture}
\newtheorem{rem}[thm]{Remark}
\newtheorem{defi}[thm]{Definition}

\newtheorem{prob}[thm]{Problem}

\newcommand{\real}{\mathbb{R}}
\newcommand{\comp}{\mathbb{C}}
\newcommand{\nat}{\mathbb{N}}
\newcommand{\im}{\text{\normalfont Im}}
\newcommand{\re}{\text{\normalfont Re}}
\newcommand{\iu}{\mathcal{UI}}
\newcommand{\id}{\mathcal{ID}}
\newcommand{\EM}{\mathcal{EM}}
\newcommand{\sym}{\text{\normalfont Sym}}
\newcommand{\MP}{\bm{\pi}}
\newcommand{\E}{\mathbb{E}}
\newcommand{\U}{\mathcal{U}(\boxplus)}
\newcommand{\F}{\mathcal{F}}
\newcommand{\M}{\mathcal{M}}

\begin{document}
\title{Classical Scale Mixtures of Boolean Stable Laws}

\author{Octavio Arizmendi \\ Department of Probability and Statistics, CIMAT, \\ Guanajuato, Mexico \\  Email:
octavius@cimat.mx  \\ \\ Takahiro Hasebe\footnote{Supported by Marie Curie Actions -- International Incoming Fellowships (Project 328112 ICNCP) at University of Franche-Comt\'e; supported also by Global COE program ``Fostering top leaders in mathematics -- broadening the core and exploring new ground'' at Kyoto university.}\\ 
Department of Mathematics, Hokkaido University \\ Kita 10, Nishi 8, Kitaku, Sapporo 060-0810, Japan. \\ Email:
thasebe@math.sci.hokudai.ac.jp } 
\date{\today}
\maketitle

\maketitle
\begin{abstract}
We study Boolean stable laws, $\mathbf{b}_{\alpha,\rho}$, with stability index $\alpha$ and asymmetry parameter $\rho$. We show that the classical scale mixture of $\mathbf{b}_{\alpha,\rho}$ coincides with a free mixture and also a monotone mixture of $\mathbf{b}_{\alpha,\rho}$. 
For this purpose we define the multiplicative monotone convolution of probability measures, one is supported on the positive real line and the other is arbitrary. 

We prove that any scale mixture of $\mathbf{b}_{\alpha,\rho}$ is both classically and freely infinitely divisible for $\alpha\leq1/2$ and also for some $\alpha>1/2$. Furthermore, we show the multiplicative infinite divisibility of $\mathbf{b}_{\alpha,1}$ with respect classical, free and monotone convolutions. 

Scale mixtures of Boolean stable laws include some generalized beta distributions of second kind, which turn out to be both classically and freely infinitely divisible. One of them appears as a limit distribution in multiplicative free laws of large numbers studied by Tucci, Haagerup and M\"oller.

We use a representation of $\mathbf{b}_{\alpha,1}$ as the free multiplicative convolution of a free Bessel law and a free stable law to prove a conjecture of Hinz and M{\l}otkowski regarding the existence of the free Bessel laws as probability measures. The proof depends on the fact that $\mathbf{b}_{\alpha,1}$ has free divisibility indicator 0 for $1/2<\alpha$. 
\end{abstract}

Mathematics Subject Classification 2010: 46L54, 60E07

Keywords: Free convolution, Boolean stable law, infinite divisibility, mixtures, free Bessel law

\tableofcontents

\section{Introduction}

% HCM is always in the image of Markov transform. Then maybe this is a good class also for FID.
%Markov transform of Boolean is Monotone. What about Markov transform of Monotone? 1/2 is a shift?

In this paper we study different aspects of classical scale mixtures of Boolean stable laws $\mathbf{b}_{\alpha,\rho}$ including classical and free infinite divisibility, unimodality and relation with other distributions such as classical stable laws, free stable laws and free Bessel laws.

We denote respectively by $\ast,\boxplus$ the classical and free additive convolutions, and by $\circledast, \boxtimes$ the classical and free multiplicative convolutions.  An important class of measures in connection with the study of limit laws is the class of infinitely divisible distributions. A probability measure $\mu$ is said to be \textbf{(classically) infinitely
divisible} (or \textbf{ID} for short) if, for every natural number $n$, there exists a probability measure $\mu_n$ such
that $$
\mu = \underbrace{\mu_n*\mu_n*\cdots\mu_n}_{n \text{~times}}.
$$
In the same way,  in free probability a measure $\mu$ is said to be \textbf{freely infinitely divisible} (or \textbf{FID} for short) if, for every natural number $n$, there exists a probability measure $\mu_n$ such that
$$
\mu = \underbrace{\mu_n\boxplus\mu_n\boxplus\cdots\boxplus\mu_n}_{n\text{~times}}. 
$$
We denote by $\id(\ast)$ the class of all ID distributions on $\mathbb{R}$ and by $\id(\boxplus)$ the class of all FID distributions on $\mathbb{R}$. 

The Boolean stable law $\mathbf{b}_{\alpha,\rho}$ appears as the stable distribution for Boolean independence \cite{S-W}. The positive one $\mathbf{b}_{\alpha,1}$ is the law of quotient of identically distributed, independent positive $\alpha$-stable random variables. The density is given by
 \begin{equation*}
 \label{density}\frac{\frac{1}{\pi}\sin(\alpha\pi)x^{\alpha-1}}{x^{2\alpha}+2\cos(\alpha\pi)x^\alpha+1},~~ x>0.
 \end{equation*}
   The authors studied these measures in \cite{AHb} in relation to classical and free infinite divisibility, proving that the Boolean stable law is FID for $\alpha \leq 1/2$ or $1/2<\alpha\leq2/3,  2-1/\alpha \leq \rho \leq 1/\alpha-1$. Moreover the positive Boolean stable law for $\alpha \leq1/2$ is both ID and FID. Note that Jedidi and Simon showed that it is HCM, more strongly than ID \cite{JS}. The positive Boolean stable law was the first nontrivial continuous family of measures which are ID and FID. 

Our main result is in Section \ref{Main section}. We extend the results in \cite{AHb} to \emph{classical} scale mixtures of Boolean stable laws, giving a large class of probability measures being ID and FID. 
\begin{thm}\label{main} Let $B_{\alpha,\rho}$ be a random variable following the Boolean stable law $\mathbf{b}_{\alpha,\rho}$,  and let $X$ be \textbf{any} nonnegative random variable classically independent of $B_{\alpha,\rho}$. 
If $\alpha \in (0,1/2]$ or if $\alpha \leq 2/3, \rho =1/2$, then the law of $X B_{\alpha,\rho}$ is in $\id(\boxplus)\cap \id(\ast)$. 
\end{thm}
The proof is given separately for ID and FID parts. We show in Theorem \ref{classical ID} that the law of $X B_{\alpha,\rho}$ is ID if: (a) $\alpha\leq1/2$; (b) $\alpha \leq 1, \rho=1/2$. (a), (b) may not be necessary conditions for $X B_{\alpha,\rho}$ being ID. 
 The proof depends on mixtures of exponential distributions for $\alpha\leq 1/2$ and mixtures of Cauchy distributions for $\rho=1/2$. For the free part, we show in Theorem \ref{FIDB} that the law of $X B_{\alpha,\rho}$ is FID for any $X\geq0$ if and only if: (i) $\alpha \leq 1/2$; (ii) $1/2<\alpha\leq 2/3, 2-1/\alpha\leq\rho\leq 1/\alpha-1$. The proof is based on complex analysis; we show that the Voiculescu transform has an analytic extension defined in $\mathbb{C}^+$ taking values in $\mathbb{C}^-\cup \mathbb{R}$ (see \cite{Be-Vo}). When $B_{\alpha,\rho}$ is symmetric or positive, we give a simpler proof by using the identities  
\begin{align}
&\mathbf {b}_{\alpha,1}= \MP^{\boxtimes \frac{1-\alpha}{\alpha}}\boxtimes \mathbf{f}_{\alpha,1}, &\alpha \in (0,1]&, \label{01}\\
 &\mathbf{b}_{\alpha,1/2}=\MP^{\boxtimes\frac{2-\alpha}{2\alpha}}\boxtimes \sym\!\left(\sqrt{\mathbf{f}_{\alpha/2,1}}\right),&\alpha \in(0,2]&, 
\end{align}
where $\mathbf{f}_{\alpha,\rho}$ is a free stable law and $\MP$ is a free Poisson. See Section \ref{Preliminary} for the other notations. 
We also show the multiplicative infinite divisibility for $\mathbf{b}_{\alpha,1}$ with respect to classical, free and monotone convolutions. 

In Subsections \ref{subsection basic Boole} and \ref{section identities}, we establish a lot of identities involving classical, Boolean and free stable laws, and multiplicative classical $\circledast$, free $\boxtimes$ and monotone $\circlearrowright$ convolutions. For this purpose, we define the multiplicative monotone convolution of two probability measures, one is supported on $[0,\infty)$ and the other is on $\real$, in Subsection \ref{subsection basic monotone}. 
The most outstanding result in this context is the following identity:
\begin{align*} 
&\mu^{1/\alpha}\circledast\mathbf{b}_{\alpha,\rho} =\mu^{\boxtimes 1/\alpha}\boxtimes\mathbf{b}_{\alpha,\rho}, \quad\text{$\mathbf{b}_{\alpha,\rho}$ being positive or symmetric}. 
%&\mu^{1/\alpha}\circledast\mathbf{b}_{\alpha,\rho} =\mu\circlearrowright\mathbf{b}_{\alpha,\rho}. 
\end{align*}
The measure $\mu^{1/\alpha}$ is the law of $X^{1/\alpha}$ when $X$ follows the law $\mu$. This identity gives us a direct connection between classical multiplication and free multiplication, and it suggests the importance of $\mathbf{b}_{\alpha,\rho}$. In Subsection \ref{Connection}, we compare classical, free and Boolean stable laws and observe similarities between them.  

Examples of random variables $X B_{\alpha,\rho}$ as in Theorem \ref{main} are provided in Section \ref{Examples}. We give new probability measures which are both ID and FID, including the generalized beta distributions of the second kind with densities
\begin{align*}
&c_{\alpha,\beta}\cdot\frac{x^{\alpha-3/2}}{(x^{\alpha\beta}+1)^{1/\beta}}1_{(0,\infty)}(x),&\alpha \in(1/2, 1], \beta\in(0,1/\alpha]&, \\
&\frac{\beta x^{\beta-1}}{(x^\beta+1)^2}1_{(0,\infty)}(x),  &\beta \in (0,1/2]&,
\end{align*}
see Proposition \ref{gen beta1} and Example \ref{gen beta2}.  Moreover, these measures are HCM (see \cite{B92}). 
 We compute the limit distributions in free multiplicative laws of large numbers \cite{T10,HM}
$$
\lim_{n\to\infty}(\mu^{\boxtimes n})^{1/n}
$$
 by taking $\mu$ to be the law of $X B_{\alpha,1}$. The limit distribution is again a scale mixture of Boolean stable laws, but now with stability index $\alpha/(1-\alpha)$. We consider the free Jurek class which is the free analogue of Jurek class \cite{J85}. The law of $X B_{\alpha,\rho}$ belongs to the free Jurek class for $\alpha \lesssim  0.42, \rho=1$ and for $\alpha\leq1/2, \rho=1/2$.

Free Bessel laws, introduced in Banica et al.\ \cite{BBCC},  are measures $\MP_{st}=(1-t)\delta_0+D_t((\MP^{\boxtimes s})^{\boxplus 1/t})$ for $s>0, 0<t \leq1$, where $\MP$ is the free Poisson with mean 1. 
Note that  $\MP_{st}$ is a probability measure since $\MP$ is $\boxtimes$-infinitely divisible. It is also known that the parameters may be extended to $s\geq1, t>0$. The question of whether one can extend these parameters for $0<s<1,t>1$ was raised in \cite{BBCC}. 
Later, from considerations of moments,  Hinz and M{\l}otkowski \cite{HinMlo} conjectured that $(\MP^{\boxtimes s})^{\boxplus t}$ is not a probability measure for $0<s,t<1$. 
 In the last part of the paper, %we study a relation between free Bessel laws, positive stable distributions and free stable laws $\mathbf{f}_{\alpha,\rho}$. 
  we give an answer to the conjecture of Hinz and M{\l}otkowski using the representation (\ref{01}) and the free divisibility indicator, and we then settle the question of  the existence of free Bessel laws as a corollary. 
\begin{thm}\label{Thm3} Let $s,t>0$ and let $\tilde{\MP}_{st}=(\MP^{\boxtimes s})^{\boxplus t}$. Then 
$\tilde{\MP}_{st}$ is a probability measure if and only if $\max(s,t)\geq1$. In other words, the sequence of Fuss-Narayana polynomials given by
$$
\tilde{m}_0(s,t)=1,\qquad \tilde{m}_n(s,t)=\sum_{k=1}^n\frac{t^k}{n}\binom{n}{k-1}\binom{n s}{n-k}, \quad n \geq1
$$
is a  sequence of moments of a probability measure on $\mathbb{R}$ if and only if $\max(s,t)\geq1$. 
In particular, the free Bessel law $\MP_{st}$ is a probability measure if and only if $(s,t) \in (0,\infty) \times (0,\infty)- (0,1)\times(1,\infty)$. \end{thm}

 %The paper is organized as follows ???
 % In the appendix we give an alternative proof for the free infinite divisibility of $b_\alpha$. 
%This  relies on our recent with Hasebe and Sakuma on free regular measures [?] (see section ?? for definition). 

%In Section 5 we give various examples of measures which appear ar Boolean Mixtures. Finally, we study free Bessel Laws in Section 6. 

\section{Preliminaries}\label{Preliminary}

%In this section we describe the preliminaries needed for this paper.  For further details we recommend [?][?] and [?].

\subsection{Notations}\label{Notation}

We collect basic notations used in this paper. 
\begin{enumerate}[\rm(1)] 

\item $\nat=\{1,2,3,\dots\}$ is the set of strictly positive natural numbers. 
\item $\mathcal{P}$ is the set of (Borel) probability measures on $\real.$

\item $\mathcal{P}_+$ is the set of probability measures on $\real_+=[0,\infty).$

\item $\mathcal{P}_s$ is the set of symmetric probability measures on $\real$. 

\item For $a \in \real$, we denote by $D_a\mu$ the dilation of a probability measure $\mu$, i.e.\ if a random variable $X$ follows $\mu$, then $D_a\mu$ is the law of $a X$. 

\item $\comp^+, \comp^-$ denote the complex upper half-plane and the lower half-plane, respectively. 

\item For $D \subset \comp\setminus\{0\}$, the set $D^{-1}$ denotes the image of $D$ by the map $z\mapsto z^{-1}$. 

\item For $\theta_1< \theta_2$ such that $\theta_2-\theta_1\leq 2\pi$,  $\comp_{(\theta_1,\theta_2)}$ is the sector $\{ r e^{i\theta}: r>0, \theta \in(\theta_1,\theta_2)\}$. 

\item For $\alpha,\beta>0$, $\Gamma_{\alpha,\beta}$ is the truncated cone $\{z\in\comp^+: |\text{Re}(z)| < \alpha\text{Im}(z), |z|>\beta\}$. 

%\item For $\alpha>0$, $\Gamma_{\alpha}$ is the cone $\{z\in\comp^+: |\text{Re}(z)| < \alpha \text{Im}(z) \}$. 

%\item $\Beta_{p,q}$ is the beta distribution on $[0,1]$ having the density 
%$$
%\frac{1}{B(p,q)} x^{p-1}(1-x)^{q-1}, 
%$$
%where $B(p,q)$ is the beta function. 
\item For $p \in \real$ and $\mu\in\mathcal{P}_+$, let $\mu^p$ be the push-forward of $\mu$ by the map $x\mapsto x^p$. If $p$ is an integer, we define $\mu^p$ for any $\mu\in\mathcal{P}$.   If $p\leq0$, we define $\mu^{p}$ only when $\mu(\{0\})=0$. 
We may use the notation $\sqrt{\mu}$ instead of $\mu^{1/2}$. 

%\item For $\mu \in \mathcal{P}$, the measure $|\mu| \in \mathcal{P}_+$ is the push-forward of $\mu$ by the map $x \mapsto |x|$. 

\item For $\mu \in \mathcal{P}_+$, the measure $\sym(\mu)$ is the symmetrization $\frac{1}{2}(\mu(dx)+\mu(-dx))$. 

\item For $z\in \comp\setminus(- \real_+)$, $\arg z$ is the argument of $z$ taking values in $(-\pi,\pi)$. 

\item For $p \in \real$, the power $z\mapsto z^p$ denotes the principal value $|z|^p e^{i p \arg z}$ in $\comp \setminus(-\real_+)$.  

\item For $\theta_1<\theta_2$ such that $\theta_2-\theta_1\leq 2\pi$,  $\arg_{(\theta_1,\theta_2)} z$ denotes the argument  of $z\in \comp_{(\theta_1,\theta_2)}$ taking values in $(\theta_1,\theta_2)$.

\item For $p\in\real$ and $\theta_1< \theta_2$ such that $\theta_2-\theta_1\leq 2\pi$, the power $z\mapsto (z)_{(\theta_1,\theta_2)}^p$ is defined by $|z|^p \exp(i p \arg_{(\theta_1,\theta_2)} z )$ for $z\in\comp_{(\theta_1,\theta_2)}$. 
\end{enumerate}

\subsection{Additive Convolutions}

We  briefly explain the additive convolutions from non-commutative probability used in this paper. They correspond to notions of independence coming from universal products classified by Muraki \cite{Mu2}: tensor (classical), free, Boolean and monotone independences. We omit monotone convolution since it does not appear in this paper.

\subsubsection{Classical Convolution}
Let $\F_\mu$ be the characteristic function of $\mu\in\mathcal{P}$. Then the \textbf{classical convolution} is characterized by 
$$
\F_{\mu_1\ast\mu_2}(z)=\F_{\mu_1}(z)\F_{\mu_2}(z),\qquad z\in\real, \mu_1,\mu_2\in\mathcal{P}. 
$$
 Classical convolution corresponds to the sum of (tensor) independent random variables. 
 
The \textbf{moment generating function} of $\mu\in\mathcal{P}$ is defined by $\M_\mu(z):= \F_\mu(z/i)$, $z\in i\real$. When $\mu\in\mathcal{P}_+$, the domain of $\mathcal{M}_\mu$ extends to $\{z\in\comp: \re(z)\leq0\}$. For $\mu\in\mathcal{P}$, there exists $a\in(0,\infty]$ such that $\M_\mu(z)\neq0$ in $i(-a,a)$, and then we may define the \textbf{classical cumulant transform} of $\mu\in\mathcal{P}$ by  
$$
\mathcal{C}^\ast_{\mu}(z) = \log(\M_\mu(z)), \qquad z\in i(-a,a)
$$
 such that it is continuous and $\mathcal{C}^\ast_{\mu}(0)=0$. It then follows that for some $c>0$ depending on $\mu_1,\mu_2$, 
$$
\mathcal{C}^\ast_{\mu_1*\mu_2}(z)=\mathcal{C}^\ast_{\mu_1}(z)+\mathcal{C}^\ast_{\mu_2}(z), \qquad z \in i(-c,c). 
$$
In general, $\mathcal{C}^\ast$ does not characterize the probability measure, that is, there are two distinct $\mu,\nu\in\mathcal{P}$ such that $\mathcal{C}^\ast_\mu(z)=\mathcal{C}^\ast_\nu(z)$ in some $i (-a,a)$. In particular cases such as $\mu\in\id(\ast)$, the characteristic function does not have a zero and hence the classical cumulant transform  $\mathcal{C}^\ast_\mu$ extends to a continuous function on $i\real$. Then $\mathcal{C}_\mu^\ast$ on $i\real$ (or on $-i \real_+$) uniquely determines $\mu$. 

%Similar convolutions and related transforms exist for the free, Boolean and monotone theories.

\subsubsection{Free Convolution}
Free convolution was defined in \cite{Voi85} for compactly supported probability measures and later extended in \cite{Maa} for the case of finite variance, and in \cite{Be-Vo} for the general unbounded case.
Let 
$$
G_\mu(z) = \int_{\real}\frac{\mu(dx)}{z-x},\qquad F_\mu(z)=\frac{1}{G_\mu(z)},\qquad z\in \comp\setminus \real
$$ 
be the \textbf{Cauchy transform} and the \textbf{reciprocal Cauchy transform (or $F$-transform)} of $\mu\in\mathcal{P}$, respectively. 
It was proved in Bercovici and Voiculescu \cite{Be-Vo} that there exist $\alpha,\beta,\alpha',\beta'>0$ such that $F_\mu$ is univalent in $\Gamma_{\alpha',\beta'}$ and $F_\mu(\Gamma_{\alpha',\beta'}) \supset \Gamma_{\alpha,\beta}$. Hence the left compositional inverse $F_\mu^{-1}$ may be defined in $\Gamma_{\alpha,\beta}$. 
The \textbf{Voiculescu transform} of $\mu$ is  then defined by $\phi _{\mu }\left( z\right) =F_{\mu }^{-1}(z)-z$ on the region $\Gamma_{\alpha,\beta}$ where $F^{-1}_{\mu}$ is defined. Moreover, the \textbf{free cumulant transform}  (see \cite{BNT02b}) is a variant of $\phi _{\mu }$
defined as
	$$
	\mathcal{C}_{\mu }^{\boxplus}(z)=z\phi _{\mu }\left(\frac{1}{z} \right)
	=zF_{\mu }^{-1}\left(\frac{1}{z}\right) -1,  \qquad z \in (\Gamma _{\alpha,\beta})^{-1}. 
	$$

The \textbf{free convolution} of two probability measures $\mu_1,\mu_2$ on $\mathbb{R}$ is the
probability measure $\mu_1\boxplus\mu_2$ on $\mathbb{R}$ such that 
$$
\phi_{\mu_1\boxplus\mu_2}(z) = \phi_{\mu_1}(z) + \phi_{\mu_2}(z)
$$
in a common domain $\Gamma_{\gamma,\delta}$ which is contained in the intersection of the domains of $\phi_{\mu_1}, \phi_{\mu_2}$ and $\phi_{\mu_1\boxplus\mu_2}$. 
Free convolution corresponds to the sum of free random variables \cite{Be-Vo}. 

For any $t \geq1$ and any $\mu\in\mathcal{P}$, there exists a measure $\mu^{\boxplus t} \in\mathcal{P}$ which satisfies $\phi_{\mu^{\boxplus t}}(z) = t \phi_\mu(z)$ in a common domain \cite{NiSp96}.  

%$\mu_a\boxplus\mu_b=\mu_{a+b}$, for $a$ and $b$ free random variables.  

\subsubsection{Boolean Convolution}
The \textbf{Boolean convolution} \cite{S-W} of two probability measures $\mu_1 , \mu_2$ on $\mathbb{R}$ is defined as
the probability measure $\mu_1\uplus\mu_2$ on $\mathbb{R}$ such that
$$
\eta_{\mu_1\uplus\mu_2}(z) = \eta_{\mu_1}(z) + \eta_{\mu_2}(z) ,~~~z\in\mathbb{C}^-, 
$$ 
where the \textbf{$\eta$-transform} (or Boolean cumulant transform) is defined by  
\begin{equation}\label{etaf}
\eta_\mu(z)=1-zF_\mu\left(\frac{1}{z}\right),\qquad z\in\comp^-.  
\end{equation}
Boolean convolution corresponds to the sum of Boolean independent random variables. Such an operator-theoretic model was constructed in \cite{S-W} for bounded operators and  in \cite{Franz} for unbounded operators. 

For any $t \geq 0$ and any $\mu\in\mathcal{P}$, there exists a measure $\mu^{\uplus t} \in\mathcal{P}$ which satisfies $\eta_{\mu^{\uplus t}}(z) = t \eta_\mu(z)$ in $\comp^-$ \cite{S-W}.

%\subsubsection{Monotone Convolution}
%The \textbf{monotone convolution} $\mu_1\rhd \mu_2$ of probability measures $\mu_1 , \mu_2$ on $\mathbb{R}$ was defined in \cite{Mu3} 
%such that 
%$$
%F_{\mu_1\rhd\mu_2}(z) = F_{\mu_1}(F_{\mu_2}(z)) ,~~~z\in\mathbb{C}^+.
%$$ 
%Monotone convolution corresponds to the sum of monotone independent random variables. Such an operator-theoretic model was constructed in \cite{Mu3} for bounded operators and  in \cite{Franz} for unbounded operators.

\subsection{Multiplicative Convolutions}

\subsubsection{Multiplicative Classical Convolution}
The \textbf{multiplicative classical convolution} $\mu_1\circledast\mu_2$ of $\mu_1,\mu_2\in\mathcal{P}$ is defined by 
$$
\int_{\real}f(x)(\mu \circledast \nu)(dx) = \int_{\real}f(x y)\mu(dx)\nu(dy)
$$ 
for any bounded continuous function $f$ on $\real$.  
The measure $\mu_1 \circledast \mu_2$ corresponds to the distribution of  $X Y$, where $X$ and $Y$ are independent random variables with distributions $\mu_1$ and $\mu_2$, respectively.

\subsubsection{Multiplicative Free Convolution}
For probability measures $\mu_1 \in \mathcal{P}_+,\mu_2\in\mathcal{P}$, the \textbf{multiplicative free convolution} $\mu_1 \boxtimes \mu_2 \in\mathcal{P}$ is defined as the  distribution of $\sqrt{X}Y\sqrt{X}$, where $X \geq 0$ and $Y$ are free random variables with distributions $\mu_1$ and $\mu_2$, respectively.
Multiplicative free convolution was introduced in \cite{Voi87} for compactly supported probability measures, and then extended in \cite{Be-Vo} for non compactly supported probability measures.  

%To investigate multiplicative free convolution, important transforms are \textbf{$\psi$- and $\eta$-transforms}: 
%\[
%\psi_\mu(z)=\int_{\real} \frac{x z}{1-x z}\mu(dx ),\qquad \eta_\mu(z)=\frac{\psi_\mu(z)}{1+\psi_\mu(z)},\qquad z\in\comp^-, \quad \mu\in\mathcal{P}. 
%\]
%The transform $\eta_\mu$ for $\eta$ is characterized as follows \cite{Bel3}.
%\begin{prop}\label{prop0}
% Let $\eta: \comp \backslash \real_+ \to \comp$ be an analytic map satisfying $\eta(\overline{z}) = \overline{\eta(z)}$ and $\eta \neq 0$. Then the following properties are equivalent. \\
% $\eta =\eta_\mu$ for a probability measure $\mu \in \mathcal{P}(\real_+)$, $\mu \neq \delta_0$. \\
%$\eta(-0) = 0$ and $\arg \eta(z) \in [\arg z, \pi)$ for any $z \in \comp^+$.
%\end{prop}
Suppose that $\delta_0\neq \mu \in \mathcal{P}_+$ (resp.\ $\delta_0\neq \mu \in \mathcal{P}_s$). The function $\eta_\mu$ is univalent around $(-\infty, 0)$ (resp.\ $i (-\infty,0)$) taking values in a neighborhood of the interval $(1-(\mu(\{0\}))^{-1},0)$ (we understand that $(\mu(\{0\}))^{-1}=\infty$ if $\mu(\{0\})=0$), so that 
one may define the compositional inverse $\eta_\mu^{-1}$ and then the \textbf{$\Sigma$-transform}
\[
\Sigma_\mu(z) :=\frac{\eta_\mu^{-1}(z)}{z}, \qquad z\in (1-(\mu(\{0\}))^{-1},0). 
\]
 Multiplicative free convolution $\boxtimes$ is characterized by the multiplication of $\Sigma$-transforms: 
\begin{equation}\label{eq324}
\Sigma_{\mu_1 \boxtimes \mu_2}(z) = \Sigma_{\mu_1}(z) \Sigma_{\mu_2}(z), \qquad \mu_1 \in \mathcal{P}_+, \mu_2\in\mathcal{P}_+ \text{~or~}\mu_2\in\mathcal{P}_s
\end{equation}
in the common interval $(-\beta, 0)$, provided $\mu_1 \neq \delta_0 \neq \mu_2$. The case $\mu_1,\mu_2 \in \mathcal{P}_+$ was proved in \cite{Be-Vo} and the case $\mu_1\in\mathcal{P}_+, \mu_2\in\mathcal{P}_s$ was proved in \cite{APA}. When $\mu_1\in\mathcal{P}_+,\mu_2\in\mathcal{P}$ and they have compact supports, (\ref{eq324}) was proved in a neighborhood of 0 in \cite{Voi87} and \cite{RaSp07}. 
In the most general case $\mu_1\in\mathcal{P}_+$ and $\mu_2\in\mathcal{P}$, it is still an open problem to define an appropriate $S$-transform $S_{\mu_2}$ and to prove (\ref{eq324}). 

Instead of the $\Sigma$-transform, often used to calculate multiplicative free convolution is the \textbf{$S$-transform}: 
\begin{equation}
S_\mu(z)=\Sigma_\mu\left(\frac{z}{1+z}\right),\qquad z\in (-1+\mu(\{0\}),0). 
\end{equation}

If $\mu \in \mathcal{P}_+$, then a convolution power $\mu^{\boxtimes t} \in \mathcal{P}_+$, satisfying $\Sigma_{\mu^{\boxtimes t}}(z)=(\Sigma_\mu(z))^t$, is well defined for any $ t \geq 1$ (\cite{Bel3}).
A probability measure $\mu\in\mathcal{P}_+$ is said to be \textbf{$\boxtimes$-infinitely divisible} if for any $n \in \nat$, there is $\mu_n$ on $\real_+$ such that $\mu=\mu_n ^{\boxtimes n} = \mu_n \boxtimes \cdots  \boxtimes \mu_n$.

For $\mu\in\mathcal{P}_+, \nu\in\mathcal{P}$, the identity 
 \begin{equation}  \label{dilation}
D_{1/t}(\mu ^{\boxplus t}\boxtimes \nu ^{\boxplus t})=(\mu \boxtimes \nu)^{\boxplus t},\qquad t\geq 1 
\end{equation}
was proved in \cite[Proposition 3.5]{BN08}. 

Using the $S$-transform, it was proved in \cite{APA} that, for $\mu\in\mathcal{P}_+$ and $\nu\in\mathcal{P}_s$, the following relation holds: 
\begin{equation}\label{square}
(\mu\boxtimes\nu)^2=\mu\boxtimes\mu\boxtimes\nu^2.  
\end{equation} 

Assume $\mu \in \mathcal{P}_+$ and $\mu(\{0\})=0$. The following formula is known \cite[Proposition 3.13]{HaSch}:  
\begin{equation}\label{inverse S}
S_{\mu^{-1}}(z) = \frac{1}{S_\mu(-z-1)}, \qquad z\in (-1,0). 
\end{equation}

\subsection{Free Infinite Divisibility}
%\begin{defi}
%Let $\mu$ be a probability measure in $\mathbb{R}$. We say that $\mu$ is \textbf{freely infinitely divisible (or FID)}, if for all $n $, there exists a probability measure $\mu _{n}$ such that 
%\begin{equation}
%\mu =\underset{n\text{ times}}{\underbrace{\mu _{n}\boxplus \mu _{n}\boxplus....\boxplus \mu _{n}}}\text{.}  \label{InfDivLib}
%\end{equation}%
%We denote by $\id(\boxplus) $ the class of such measures. 
%\end{defi}
%If $\mu \in I^\boxplus $ we also say that $\mu$ is FID or infinitely divisible with respect to $\boxplus$.
\subsubsection{Characterization, L\'evy-Khintchine Representation}

Recall that a probability measure $\mu$ is ID if and only if its classical cumulant transform $\mathcal{C}_\mu^\ast$ has the L\'{e}vy-Khintchine representation (see e.g.\ \cite{Sa99})
\begin{equation}
\mathcal{C}_\mu^\ast(z)=\eta z+ \frac{1}{2}a z^2+\int_{
\mathbb{R}}(e^{z t}-1- z t 1_{\left[ -1,1\right] }(t) )\,\nu(dt),
\text{ \ \ }z\in i\mathbb{R},  \label{levykintchine clasica}
\end{equation}
where $\eta \in \mathbb{R},$ $a\geq 0$ and $\nu $ is a L\'{e}vy measure on $\mathbb{R}$, that is,
$\int_{\mathbb{R}}\min (1,t^{2})\nu (dt)<\infty $ and $\nu (\{0\})=0$. If this representation exists, the
triplet $(\eta ,a,\nu )$ is unique and is called the classical characteristic triplet of $\mu $.

A FID measure has a free analogue of the L\'{e}vy-Khintchine representation. 
\begin{thm}[Voiculescu \cite{Voi86}, Maassen \cite{Maa}, Bercovici \& Voiculescu \cite{Be-Vo},  Barndorff-Nielsen \& Thorbj{\o}rnsen \cite{BNT02b}] \label{thmBV93}
For a probability measure $\mu$ on $\real$, the following are equivalent. 
\begin{enumerate}[\rm(1)] 
\item $\mu$ belongs to $\id(\boxplus)$. 
\item $-\phi_\mu$ extends to a Pick function, i.e.\ an analytic map of $\comp^+$ into $\comp^+ \cup \real$. 
\item For any $t>0$, there exists a probability measure $\mu^{\boxplus t}$ with the property $\phi_{\mu^{\boxplus t}}(z) = t\phi_\mu(z).$
\item A probability measure $\mu $ on $\mathbb{R}$ is FID if and only if there are $\eta_\mu \in 
\mathbb{R},$ $a_\mu\geq 0$ and a L\'{e}vy measure $\nu_\mu$ on $\mathbb{R}$ such that 
\begin{equation}
\mathcal{C}^\boxplus_{\mu}(z)=\eta_\mu z+a_\mu z^{2}+\int_{\mathbb{R}}\left( \frac{1}{1-zt}-1-tz1_{\left[ -1,1\right] }\left( t\right) \right) \nu_\mu(dt) ,\quad z\in \comp^-.  \label{levykintchine libre}
\end{equation}
The triplet $(\eta_\mu,a_\mu,\nu_\mu)$ is unique and is called the \textbf{free characteristic triplet} of $\mu$ and $\nu_\mu$ is called the L\'evy measure of $\mu$. 
\end{enumerate}
\end{thm}
%Two important examples of FID measures that we will use often in this paper. 
%On one hand, the standard Wigner semicircle law $w$, with free characteristic triplet $(0,1,0 )$ and density 
%\begin{equation*}
%\frac{1}{2\pi }\sqrt{4-x^2}\mathrm{d}x,\quad -2<x<2. 
%\end{equation*}
An important FID distribution in this paper is the free Poisson law $\MP$, also known as the Marchenko-Pastur law, with free characteristic triplet $(1,0,\delta_1 )$ and density 
\begin{equation*}
\frac{1}{2\pi }\sqrt{\frac{4-x}{x}}\mathrm{d}x,\quad 0<x<4.
\end{equation*}
The free Poisson distribution $\MP$ is infinitely divisible both with respect to $\boxtimes$ and $\boxplus$.

\subsubsection{Compound Free Poisson Distribution}
Suppose that $\sigma \in \id(\boxplus)$ does not have a semicircular component ($a_\sigma=0$) and that the L\'evy measure $\nu_\sigma$ in  (\ref{levykintchine libre}) satisfies $\int_{\mathbb{R}_+}\min(1,t)\nu_\sigma(dt) < \infty$.  Then the L\'evy-Khintchine representation reduces to 
\begin{equation}\label{eq00}
\mathcal{C}^{\boxplus}_{\sigma }(z)=\eta_\sigma' z+\int_{\mathbb{R}}\left( \frac{1%
}{1-zt}-1\right) \nu_\sigma \left(d t\right) ,\quad z\in \comp^-,
\end{equation}
where $\eta_\sigma' \in \real$. 
%The measure $\sigma$ is said to be a \textbf{free regular infinitely divisible (or free regular, for short) distribution} if $\eta_\sigma' \geq 0$ and $\nu_\sigma\left((-\infty,0]\right)=0$.  A probability measure $\sigma$ is free regular if and only if $\mu^{\boxplus t} \in \mathcal{P}_+$ for all $t>0$.
%An important class of free regular measures comes from compound free Poisson distributions. 
The measure $\sigma$ is called the {\bf compound free Poisson distribution} (\cite{Sp98}) with rate $\lambda$ and jump distribution $\rho$ if the drift term $\eta_\sigma'$ is zero and the L\'evy measure $\nu_\sigma$ is $\lambda \rho$ for some $\lambda>0$ and a probability measure $\rho$ on $\real$. To clarify these parameters, we denote $\sigma = \MP(\lambda,\rho)$.  
\begin{rem} \label{rem free poisson}
\begin{enumerate}[\rm(1)]
\item The Marchenko-Pastur law $\MP$ is a compound free Poisson with rate $1$ and jump
distribution $\delta _{1}$. 
\item For any $\nu\in\mathcal{P}$, the compound free Poisson $\MP (1,\nu)$ coincides with the free multiplication $\MP \boxtimes\nu$ (\cite{NiSp06}). 
\end{enumerate}
\end{rem}

\subsubsection{Free Divisibility Indicator}
A one-parameter family of maps $\{\mathbb{B}_t \}_{t \geq 0}$ on $\mathcal{P}$, introduced by Belinschi and Nica \cite{BN08}, is defined by  
\[
\mathbb{B}_t(\mu) = \Big(\mu^{\boxplus (1+t)}\Big) ^{\uplus\frac{1}{1+t}}. 
\]
The family $\{\mathbb{B}_t \}_{t \geq 0}$ is a composition semigroup and, moreover, each map $\mathbb{B}_t$ is a homomorphism regarding multiplicative free convolution: $\mathbb{B}_t(\mu \boxtimes \nu) = \mathbb{B}_t (\mu) \boxtimes \mathbb{B}_t(\nu)$ for probability measures $\mu\in\mathcal{P}_+, \nu\in\mathcal{P}$. 

 Let $\phi(\mu)$ denote the free divisibility indicator defined by 
\begin{equation}
\phi(\mu):=\sup \{t \geq 0: \mu \in \mathbb{B}_t(\mathcal{P}) \},  
\end{equation}
which has another expression \cite{AH2}
\begin{equation}\label{free div ind char}
\phi(\mu)= \sup\{t\geq0: \mu^{\uplus t} \in \id(\boxplus) \}. 
\end{equation}
For any $\mu\in\mathcal{P}$ and $0\leq s\leq \phi(\mu)$, Belinschi and Nica proved that a probability measure $\nu_s$ uniquely exists such that $\mathbb{B}_s(\nu_s)=\mu$. Therefore, the definition of $\mathbb{B}_t(\mu)$ may be extended for $0\geq t \geq -\phi(\mu)$ by setting $\mathbb{B}_t(\mu)=\nu_{-t}$. The indicator $\phi(\mu)$ satisfies the following properties \cite{BN08}. 
\begin{thm}\label{thm001}
\begin{enumerate}[\rm(1)]
\item\label{div i} $\mu^{\boxplus t}$ exists if and only if $\phi(\mu)\geq1-t$. 
\item\label{div ii} $\mu$ is FID if and only if $\phi(\mu) \geq 1$. 
\item $\phi(\mathbb{B}_t(\mu))$ can be calculated as 
\begin{equation*}\label{eq7}
\phi(\mathbb{B}_t(\mu)) = \phi(\mu)+t,\qquad t \geq -\phi(\mu). 
\end{equation*}
\end{enumerate}
\end{thm}

%The behavior of the free divisibility indicator with respect to additive free and Boolean convolutions was characterized in \cite{AH2}.
%\begin{thm}\label{Bozejko} Let $\mu$ be a probability measure on $\real$. Then
% $\phi(\mu^{\uplus t}) = \frac{1}{t}\phi(\mu)$ for $t > 0$ and $\phi(\mu^{\boxplus t})-1 = \frac{1}{t}(\phi(\mu)-1)$ for $t > \max \{1- \phi(\mu),0\}$.
%\end{thm}

More information on $\mathbb{B}_t(\mu)$ and $\phi(\mu)$ is found in \cite{BN08,AH2,Hu}.

\subsection{Stable Distributions}\label{subsection stable}
Let $\mathfrak{A}$ be the set of \textbf{admissible parameters}: 
$$
\mathfrak{A}=\{(\alpha,\rho): \alpha \in(0,1], \rho \in [0,1]\} \cup \{(\alpha,\rho): \alpha\in(1,2], \rho \in [1-\alpha^{-1}, \alpha^{-1}]\}. 
$$ 
\begin{defi} Assume that $(\alpha,\rho)$ is admissible. 
The classical $\mathbf{n}_{\alpha,\rho}$ (see e.g.\ \cite{Sa99}), Boolean $\mathbf{b}_{\alpha,\rho}$ \cite{S-W}, free $\mathbf{f}_{\alpha,\rho}$ \cite{Be-Vo,BP99} and monotone $\mathbf{m}_{\alpha,\rho}$ \cite{H2,W} \textbf{strictly stable distributions} are defined, respectively, by their classical cumulant, $\eta$, free cumulant and $F$ transforms as follows:
\begin{align} 
&\mathcal{C}^\ast_{\mathbf{n}_{\alpha,\rho}}(z)=- (e^{i \rho\pi}z)^{\alpha},  &z\in i(-\infty,0)&; \label{classical stable} \\
&\eta_{\mathbf{b}_{\alpha,\rho}}(z)= - (e^{i\rho\pi}z)^{\alpha}, &z\in \comp^-&; \label{etaboole}\\
&\mathcal{C}^\boxplus_{\mathbf{f}_{\alpha,\rho}}(z)= - (e^{i \rho\pi}z)^{\alpha}, &z\in \comp^-&;\\ 
&F_{\mathbf{m}_{\alpha,\rho}}(z)=(z^\alpha +e^{i\rho\alpha\pi})_{(0,2\pi)}^{1/\alpha},  &z\in \comp^+&. 
\end{align}
The parameters $\alpha,\rho$ are called the \textbf{stability index} and \textbf{asymmetry parameter}. 
%The powers $z\mapsto z^r$ are the principal values, while $z\mapsto (z)_{(0,2\pi)}^r$ is the power defined by $\exp(r \log|z| + i r \arg(z)_{(0,2\pi)})$ in $\comp\setminus [0,\infty)$ so that $\arg(z)_{(0,2\pi)}$ is the argument in $(0,2\pi)$.  

%In most cases we respectively denote by $\mathbf{b}_{\alpha,1}, \mathbf{f}_\alpha^+, \mathbf{m}_\alpha^+$ the measures $\mathbf{b}_{\alpha,1}, \mathbf{f}_{\alpha,1}, \mathbf{m}_{\alpha,1}$ for the clearer meaning of notations. Similarly, the notations $\mathbf{b}_{\alpha,1}^-, \mathbf{f}_\alpha^-, \mathbf{m}_\alpha^-$ will appear. 
\end{defi}

\begin{rem}
This parametrization follows \cite{HK} (except that we include $\alpha=1$ too) and is different from \cite{BP99} to respect the correspondence with the classical stable distributions \cite{Z86}.  
%\item For $\alpha=1$, the above definitions of stable distributions are different from those in \cite{BP99} and \cite{AHb}, because we are dealing with only strictly stable distributions. For example, the definition $\phi_{\mathbf{f}_{1,\rho}}(z)= -2\rho i +\frac{2(2\rho-1)}{\pi}\log z$ is used in \cite{BP99}. 
%\item
%For $\alpha\in(1,2]$ the parametrization is different from that in \cite{BP99} and \cite{AHb} so that the distributions become continuous with respect to $(\alpha,\rho)$. 
%\item One can show that $\mathbf{b}_{\alpha,\rho}([0,\infty))=\rho.$
\end{rem}
Note that 
$$
\mathbf{n}_{1,\rho}=\mathbf{b}_{1,\rho}=\mathbf{f}_{1,\rho}=\mathbf{m}_{1,\rho},\qquad \rho\in[0,1]
$$
and it is the Cauchy distribution $\mathbf{c}_\rho$ with density 
$$
\frac{1}{\pi}\cdot \frac{\sin \pi \rho}{(x+\cos \pi \rho)^2+\sin^2\pi\rho},  
$$
with the convention $\mathbf{c}_0=\delta_{-1}$ and $\mathbf{c}_1=\delta_1$. 

The probability density functions of the Boolean (and monotone) stable laws are described in \cite{HS}. 
%and the density functions of the free stable laws are described in \cite{BP99}. %The case $\alpha\leq 1$ is in particular important in this paper. 
When $\alpha \leq1$ or when $\alpha >1, 1-1/\alpha<\rho<1/\alpha$, the Boolean stable law $\mathbf{b}_{\alpha,\rho}$ is absolutely continuous with respect to the Lebesgue measure and the density is given by 
\begin{equation} 
p^+_{\alpha,\rho}(x) 1_{(0,\infty)}(x)+ p^-_{\alpha,\rho}(x) 1_{(-\infty,0)}(x), 
\end{equation}
where 
\begin{align}
&p^+_{\alpha,\rho}(x)= \dfrac{\sin(\pi\rho \alpha)}{\pi}\cdot \displaystyle \dfrac{x^{\alpha-1}}{x^{2\alpha}+2x^\alpha \cos(\pi \rho \alpha)+1},\\
&p^-_{\alpha,\rho}(x)= \dfrac{\sin(\pi(1-\rho)\alpha)}{\pi} \cdot\displaystyle \dfrac{|x|^{\alpha-1}}{|x|^{2\alpha}+2|x|^\alpha \cos(\pi (1-\rho) \alpha)+1}.  
\end{align}
For $\alpha\in[1,2]$ and $\rho=1-1/\alpha,1/\alpha$, the measure $\mathbf{b}_{\alpha,\rho}$ has one or two atoms.

\section{Basic Results}
\subsection{Multiplicative Monotone Convolution: General Case}\label{subsection basic monotone}
The \textbf{multiplicative monotone convolution} of probability measures $\mu_1, \mu_2\in \mathcal{P}_+$ is defined as the probability measure $\mu_1 \circlearrowright  \mu_2 \in \mathcal{P}_+$ such that 
$$ 
\eta_{\mu_1 \circlearrowright \mu_2}(z) = \eta_{\mu_1}(\eta_{\mu_2}(z)) ,~~~z\in\comp^+.
$$ 
Multiplicative monotone convolution corresponds to the operator $\sqrt{X}Y\sqrt{Y}$ (not $\sqrt{Y}X\sqrt{Y}$) when $X-1$ and $Y$ are monotone independent random variables \cite{Franz} and $X,Y\geq0$. Compactly supported measures $\mu_1,\mu_2\in\mathcal{P}_+$ were considered in \cite{Ber05}  and measures $\mu_1,\mu_2\in\mathcal{P}_+$ with unbounded supports were considered in \cite{Franz}. 

From the operator model, it is natural to try to define multiplicative monotone convolution for arbitrary $\mu_1\in\mathcal{P}_+, \mu_2\in\mathcal{P}$. Actually the above operator model still works for the general case $\mu_1\in\mathcal{P}_+, \mu_2\in\mathcal{P}$ with a slight modification of proofs. 

We will define multiplicative monotone convolution in this general case in terms of complex analysis. For later use,  we extract from Belinschi and Bercovici \cite{Bel3} the following characterization of the $\eta$-transform for $\mu\in\mathcal{P}_+$. 

\begin{prop}\label{prop0} Let $\delta_0\neq \mu \in \mathcal{P}_+$. The $\eta$-transform $\eta_\mu: \comp \setminus \real_+ \to \comp$ satisfies  the following. 
\begin{enumerate}[\rm(i)]
\item\label{eta 0} $\eta_\mu(\comp^-) \subset \comp^- $ and $\eta_\mu(\comp \setminus \real_+)\subset \comp \setminus \real_+$. 
\item\label{eta i} $\eta_\mu(\overline{z}) = \overline{\eta_\mu(z)}$ for $ \comp \setminus \real_+$. 

\item\label{eta iii} $\arg(\eta_\mu(z)) \in (-\pi,\arg z]$ for any $z \in \comp^-$.

\item\label{eta ii} $\eta_\mu(z) \to 0$ as $z\to0$ non tangentially to $\real_+$. More precisely, for any $\alpha\in(0,\pi)$ we have 
$$
\lim_{z\to0,z\in \comp_{(-2\pi +\alpha,-\alpha)}}\eta_\mu(z)= 0. 
$$ 
\end{enumerate}
Conversely, if an analytic map $\eta: \comp \setminus \real_+ \to \comp$ satisfies the conditions (\ref{eta 0}) -- (\ref{eta ii}), then there exists a probability measure $\delta_0\neq \mu \in \mathcal{P}_+$ such that $\eta=\eta_\mu$. 
\end{prop}
%\begin{rem}
%The condition (\ref{eta ii}) is changed from the original statement in \cite{Bel3}. 
%\end{rem}
We characterize the $\eta$-transform of a general $\mu\in\mathcal{P}$. 
\begin{prop}\label{prop0001}
Let $\mu\neq\delta_0$ be a probability measure on $\real$. Then the $\eta$-transform $\eta_\mu:\comp^-\to\comp$ is analytic and satisfies the following.  
\begin{enumerate}[\rm(1)]
\item \label{Eta0} $\eta_\mu(\comp^-) \subset \comp \setminus \real_+$.  
\item\label{Eta2} $\arg z-\pi \leq \arg_{(-2\pi,0)}(\eta_\mu(z)) \leq \arg z$ for $z\in\comp^-$. 
\item\label{Eta1} $\eta_\mu(z)\to0$ as $z\to0, z\in\comp^-$ non tangentially to $\real$. 
\end{enumerate}
Conversely, if an analytic map $\eta: \comp^-\to\comp$ satisfies the above conditions (\ref{Eta0}), (\ref{Eta2}), (\ref{Eta1}), then there exists a probability measure $\mu \neq \delta_0$ on $\real$ such that $\eta=\eta_\mu$. 
\end{prop}
\begin{rem} The condition (\ref{Eta1}) may be replaced by the following simple one: 
\begin{equation}
\lim_{y\uparrow0}\eta_\mu(i y)=0. 
\end{equation}
For our purpose the condition (\ref{Eta1}) is more useful. 
\end{rem}
\begin{proof}
We have the formula (\ref{etaf}), and so 
\begin{equation}\label{eq661}
\eta(1/z)= \frac{z-F_\mu(z)}{z}. 
\end{equation}
If $\eta(1/z_0)=c\geq0$ for some $z_0\in\comp^+$, then $F_\mu(z_0)=(1-c)z_0$. If moreover $c>0$, then this contradicts the fact that $\text{Im}(F_\mu(z)) \geq \text{Im}(z)$ for $z\in\comp^+$. If $c=0$, then $F_\mu(z_0)=z_0$, which is possible only when $\mu=\delta_0$, a contradiction. Hence we get (\ref{Eta0}). 

We have $1/z-F_\mu(1/z) \in\comp^-\cup \real\setminus\{0\}$ for $z\in\comp^-$, and hence the condition (\ref{Eta2}) follows from the identity $\eta(z) = z(1/z-F_\mu(1/z))$. 

Since $z-F_\mu(z)=o(|z|) \in \comp^-$ as $z\to \infty, z\in\comp^-$ non tangentially to $\real$ (see \cite{Be-Vo}), we get (\ref{Eta1}). 

Conversely, suppose an analytic map $\eta: \comp^-\to\comp$ satisfies (\ref{Eta0}), (\ref{Eta2}), (\ref{Eta1}). From (\ref{Eta2}), the function $ z\eta(1/z)$ maps $\comp^+$ analytically into $\comp^-\cup\real$. Hence it has a Nevanlinna-Pick representation 
\begin{equation}\label{eta Pick}
z\eta(1/z)= -a z + b - \int_{\real}\frac{1+ x z}{x-z} \tau(dx) ,\qquad z \in\comp^-
\end{equation}
for some $a\geq0, b\in\real$ and a nonnegative finite measure $\tau$. Hence 
$$
 \eta(z) = -a + b z +  \int_{\real}\frac{z(x+ z)}{1-x z} \tau(dx) = -a +o(1)
 $$ as $z\to0,z\in\comp^-$ non tangentially to $\real$. From (\ref{Eta1}) it follows that $a=0$. From \cite[Proposition 5.2]{Be-Vo}, there exists $\mu\in\mathcal{P}$ such that $F_\mu(z)= z - z\eta(1/z)$ and hence $\eta=\eta_\mu$. The condition (\ref{Eta0}) implies that $\eta\neq0$ and hence $\mu\neq \delta_0$. 
\end{proof}
Now  we can give a complex analytic definition of $\mu_1 \circlearrowright\mu_2$. 
\begin{thm}
 Let $\mu_1 \in \mathcal{P}_+$ and $\mu_2 \in \mathcal{P}$. There exists a probability measure $\mu \in \mathcal{P}$ such that 
 $\eta_{\mu}(z)= \eta_{\mu_1}(\eta_{\mu_2}(z))$ for $z\in\comp^-$. We denote $\mu =\mu_1 \circlearrowright\mu_2$. 
\end{thm}
\begin{proof}
Proposition \ref{prop0}(\ref{eta 0}) for $\eta_{\mu_1}$ and Proposition \ref{prop0001}(\ref{Eta0}) for $\eta_{\mu_2}$ imply Proposition \ref{prop0001}(\ref{Eta0}) for $\eta_{\mu_1}\circ\eta_{\mu_2}$. 

Take any $z\in\comp^-$. Then we have 
\begin{equation}\label{eq087}
\eta_{\mu_2}(z)\in\comp^- \Rightarrow
\begin{cases}
\arg_{(-2\pi,0)} (\eta_{\mu_1}(\eta_{\mu_2}(z))) \leq \arg_{(-2\pi,0)} (\eta_{\mu_2}(z)) \leq \arg(z),\\
\arg_{(-2\pi,0)} (\eta_{\mu_1}(\eta_{\mu_2}(z)))+\pi > 0> \arg(z),  
\end{cases}
\end{equation}
where Proposition \ref{prop0}(\ref{eta iii}) and Proposition \ref{prop0001}(\ref{Eta2}) are used on the first line and Proposition \ref{prop0}(\ref{eta 0}) is used on the  second. 
On the other hand,  we have: 
\begin{equation}\label{eq088}
\begin{split}
&\eta_{\mu_2}(z) \in\comp^+\cup(-\infty,0) \\
&\qquad\Rightarrow 
\begin{cases}
\arg_{(-2\pi,0)} (\eta_{\mu_1}(\eta_{\mu_2}(z))) \leq -\pi< \arg(z),\\
\arg_{(-2\pi,0)} (\eta_{\mu_1}(\eta_{\mu_2}(z)))+\pi \geq\arg_{(-2\pi,0)} (\eta_{\mu_2}(z))+\pi > \arg(z), 
\end{cases}
\end{split}
\end{equation}
where Proposition \ref{prop0}(\ref{eta 0}) is used on the first line and  Proposition \ref{prop0}(\ref{eta iii}), (\ref{eta i}) and Proposition \ref{prop0001}(\ref{Eta2}) are used on the second. 
From (\ref{eq087}) and (\ref{eq088}), Proposition \ref{prop0001}(\ref{Eta2}) holds for  $\eta_{\mu_1}\circ \eta_{\mu_2}$. 

Finally, Proposition \ref{prop0}(\ref{eta ii}) for $\eta_{\mu_1}$ and Proposition \ref{prop0001}(\ref{Eta1}) for $\eta_{\mu_2}$ imply Proposition \ref{prop0001}(\ref{Eta1}) for  $\eta_{\mu_1}\circ \eta_{\mu_2}$. Note here that for any $\alpha\in(0,\pi/2)$, if $z\to 0$, $z\in\comp_{(-\pi+\alpha,-\alpha)}$, then $\eta_{\mu_2}(z)\in\comp_{(-2\pi+\alpha,-\alpha)}$ from Proposition \ref{prop0001}(\ref{Eta2}), and hence $\eta_{\mu_2}(z)\to0$ non tangentially to $\real_+.$
 \end{proof}

\subsection{Transforms and Identities for Stable Laws} \label{subsection basic Boole}
The following relations will be often used. 
\begin{prop}[$S, \Sigma, \eta$-transforms of $\mathbf{b}_{\alpha,\rho}, \mathbf{f}_{\alpha,\rho}, \mathbf{m}_{\alpha,\rho}, \MP$]\label{S} 
\begin{align}
%&\eta_{\mathbf{b}_{\alpha,\rho}}(z)=-(e^{i\rho\pi}z)^{\alpha}, &z\in\comp^-&,  (\alpha,\rho)\in\mathfrak{A}, \label{etaboole}\\
&\Sigma_{\mathbf{b}_{\alpha,\rho}}(z)= -e^{-i\rho\pi}(-z)^{\frac{1-\alpha}{\alpha}},&z<0&,  (\alpha,\rho)\in\mathfrak{A}, \rho\in\{0,1/2,1\},  \label{sigmaboole} \\ 
&S_{\mathbf{b}_{\alpha,\rho}}(z)=-e^{-i\rho\pi} \left(-\frac{z}{1+z}\right)^{\frac{1-\alpha}{\alpha}},&z\in(-1,0)&, (\alpha,\rho)\in\mathfrak{A}, \rho\in\{0,1/2,1\},   \label{S boole} \\
&\Sigma_{\mathbf{f}_{\alpha,\rho}}(z)=-e^{-i\rho\pi}\left(\frac{-z}{1-z}\right)^{\frac{1-\alpha}{\alpha}},&z<0&,  (\alpha,\rho)\in\mathfrak{A}, \rho\in\{0,1/2,1\}, \label{sigmafree}\\ 
&S_{\mathbf{f}_{\alpha,\rho}}(z)=-e^{-i\rho\pi}(-z)^{\frac{1-\alpha}{\alpha}},&z\in(-1,0)&,  (\alpha,\rho)\in\mathfrak{A}, \rho\in\{0,1/2,1\}, \label{S free}\\ 
&\eta_{\mathbf{m}_{\alpha,\rho}}(z)= 1-((e^{i\rho\pi}z)^\alpha+1)^{1/\alpha}, &z\in\comp^-&, (\alpha,\rho)\in\mathfrak{A}, \label{etamonotone}\\
&\Sigma_{\mathbf{m}_{\alpha,\rho}}(z)=-e^{-i\rho\pi}\frac{((1-z)^\alpha-1)^{1/\alpha}}{-z},&z<0&,  (\alpha,\rho)\in\mathfrak{A}, \rho\in\{0,1/2,1\} \label{sigmamonotone}, \\
&\Sigma_{\MP}(z) = 1-z,&z<0&, \label{sigmamp} \\
&S_{\MP}(z) = \frac{1}{1+z},&z\in(-1,0)&. \label{S MP}
\end{align} 
\end{prop}
%\begin{enumerate}[\rm(M1)]
%\item $F_{\mathbf{m}_{\alpha,\rho}}(z)=(z^\alpha +e^{i\rho\alpha\pi})^{1/\alpha},~\alpha\in(0,1),~\rho\in[0,1];$
%
%\item $F_{\mathbf{m}_{\alpha,\rho}}(z)=(z^\alpha +e^{i\rho(\alpha-2)\pi})^{1/\alpha},~\alpha\in(1,2],~\rho\in[0,1].$
%\end{enumerate}

A direct computation of densities implies the following. 
\begin{lem} Let $\alpha \in (0,2]$. Then 
\begin{equation}\label{square boole}
(\mathbf{b}_{\alpha,1/2})^2=\mathbf{b}_{\alpha/2,1}.  
\end{equation}
\end{lem}

Using the $\Sigma$- or $S$-transform, we are able to show the following. 
\begin{prop}\label{Identities}
\begin{enumerate}[\rm(1)]
\item\label{Free power1} Let $\alpha\leq1$ and $t>0$. Then 
\begin{equation}\label{free power1}
(\mathbf{b}_{\alpha,1})^{\boxtimes t} = \mathbf{b}_{\frac{\alpha}{t(1-\alpha)+\alpha},1}. 
\end{equation}
In particular, we have ${(\mathbf{b}_{1/2,1})}^{\boxtimes t}=\mathbf{b}_{1/(1+t),1},  t>0.$ 
\item\label{Multi monotone} Let $\alpha\leq1$ and $t>0$. Then 
\begin{equation}\label{multi monotone} 
{(\mathbf{b}_{\alpha,1})}^{\circlearrowright t}=\mathbf{b}_{\alpha^t,1}, 
\end{equation}
that is, $\mathbf{b}_{\alpha^s,1} \circlearrowright\mathbf{b}_{\alpha^t,1}=\mathbf{b}_{\alpha^{s+t},1}, s,t>0$. 
\item\label{Sym1} We have the representation 
\begin{equation}\label{sym1}
\mathbf{b}_{\alpha,\rho}=\pi^{\boxtimes \frac{1-\alpha}{\alpha}}\boxtimes \mathbf{f}_{\alpha,\rho}, \qquad \alpha\in(0,1], \rho\in\{0,1/2,1\}. 
\end{equation}

\item\label{Sym2} The symmetric Boolean stable law $\mathbf{b}_{\alpha,1/2}$ has the representation
\begin{equation}\label{sym2} 
\mathbf{b}_{\alpha,1/2}=\MP^{\boxtimes\frac{2-\alpha}{2\alpha}}\boxtimes \sym\!\left(\sqrt{\mathbf{f}_{\alpha/2,1}}\right),\qquad \alpha \in(0,2]. 
\end{equation}
\end{enumerate}
\end{prop}
\begin{proof}
(\ref{Free power1})\,\, Note that $\frac{\alpha}{t(1-\alpha)+\alpha} \leq 1$. The assertion follows from (\ref{sigmaboole}) and the identity
$$
\frac{(1-\alpha)t}{\alpha} = \frac{1- \frac{\alpha}{t(1-\alpha)+\alpha}}{\frac{\alpha}{t(1-\alpha)+\alpha}}. 
$$  

(\ref{Multi monotone}) is a consequence of (\ref{etaboole}). 

(\ref{Sym1})\,\, From (\ref{S free}), (\ref{S MP}) and (\ref{S boole}), we get 
\[
\begin{split}
S_{\MP^{\boxtimes  \frac{1-\alpha}{\alpha} } \boxtimes\mathbf{f}_{\alpha,\rho}}(z) 
&= S_{ \mathbf{f}_{\alpha,\rho}}(z) S_{\MP^{\boxtimes  \frac{1-\alpha}{\alpha} }}(z)= -e^{-i\rho\pi} (-z)^{\frac{1-\alpha}{\alpha}} \frac{1}{(1+z)^{\frac{1-\alpha}{\alpha}}} \\
&=-e^{-i\rho\pi} \left(-\frac{z}{1+z}\right)^{\frac{1-\alpha}{\alpha}} =   S_{\mathbf{b}_{\alpha,\rho}}(z),\qquad z \in(-1,0). 
\end{split}
\]

(\ref{Sym2})\,\, 
From (\ref{square boole}), (\ref{sym1}) and (\ref{square}), we have the representation
$$
(\mathbf{b}_{\alpha,1/2})^2=\mathbf{b}_{\alpha/2,1}=\MP^{\boxtimes\frac{1-\alpha/2}{\alpha/2}}\boxtimes \mathbf{f}_{\alpha/2,1}=\MP ^{\boxtimes\frac{2-\alpha}{2\alpha}} \boxtimes\MP^{\boxtimes\frac{2-\alpha}{2\alpha}}\boxtimes \mathbf{f}_{\alpha/2,1}=\left(\MP^{\boxtimes\frac{2-\alpha}{2\alpha}}\boxtimes \sym\!\left(\sqrt{\mathbf{f}_{\alpha/2,1}}\right)\right)^2.
$$
This means that 
$$
\mathbf{b}_{\alpha,1/2}=\MP^{\boxtimes\frac{2-\alpha}{2\alpha}}\boxtimes \sym\!\left(\sqrt{\mathbf{f}_{\alpha/2,1}}\right).
%=\MP\boxtimes\left(\MP^{\boxtimes\frac{2-3\alpha}{2\alpha}}\boxtimes \sym\!\left(\sqrt{\mathbf{f}_{\alpha/2,1}}\right)\right). 
$$
\end{proof}

\section{Scale Mixtures of Boolean Stable Laws}\label{Main section}

In this, the main section of the paper, we find identities between the classical scale mixtures, free mixtures and monotone mixtures of Boolean stable laws. We then consider the classical and free infinite divisibility of scale mixtures of Boolean stable laws.

\subsection{Definition and Properties} \label{section identities}

\begin{defi} Assume that $(\alpha,\rho)$ is admissible and $\mu\in\mathcal{P}_+$. 
\begin{enumerate}[\rm(1)] 
\item The measure $\mu\circledast\mathbf{b}_{\alpha,\rho}$ is called a \textbf{scale mixture} (or a classical scale mixture) of $\mathbf{b}_{\alpha,\rho}$.

\item The measure $\mu\boxtimes \mathbf{b}_{\alpha,\rho}$ is called a \textbf{free mixture} of $\mathbf{b}_{\alpha,\rho}$.

\item The measure $\mu\circlearrowright \mathbf{b}_{\alpha,\rho}$ is called a \textbf{monotone mixture} of $\mathbf{b}_{\alpha,\rho}$.

 \item Let $\mathcal{B}_{\alpha,\rho}:=\{\nu\circledast\mathbf{b}_{\alpha,\rho}: \nu \in \mathcal{P}_+\}$ be the set of scale mixtures of $\mathbf{b}_{\alpha,\rho}$. 
%When $\alpha \leq 1$ and $\rho=1$, we also define a subset $\mathcal{B}_{\alpha}^+:=\{\mathbf{b}_{\alpha,1} \circledast \mu: \mu \in \mathcal{P}_+\}$. 

\end{enumerate}
\end{defi}

For $\alpha \in (0,1],\rho\in[0,1]$ and $\mu\in\mathcal{P}_+$, the scale mixture $\mu\circledast\mathbf{b}_{\alpha,\rho}$ is described as follows: 
\begin{equation}\label{density mixture}
\mu\circledast\mathbf{b}_{\alpha,\rho} = \mu(\{0\})\delta_0 + (1-\mu(\{0\}))\left(p^+_{\mu,\alpha,\rho}(x) 1_{(0,\infty)}(x)+ p^-_{\mu,\alpha,\rho}(x) 1_{(-\infty,0)}(x)\right)dx, 
\end{equation}
where 
\begin{align}
&p^+_{\mu,\alpha,\rho}(x)= \dfrac{\sin(\pi\rho \alpha)}{\pi} \displaystyle\int_{(0,\infty)}\dfrac{x^{\alpha-1}t^\alpha}{x^{2\alpha}+2(xt)^\alpha \cos(\pi \rho \alpha)+t^{2\alpha}}\mu(dt),\\[10pt]
&p^-_{\mu,\alpha,\rho}(x)= \dfrac{\sin(\pi(1-\rho)\alpha)}{\pi} \displaystyle\int_{(0,\infty)}\dfrac{|x|^{\alpha-1}t^\alpha}{|x|^{2\alpha}+2|xt|^\alpha \cos(\pi (1-\rho) \alpha)+t^{2\alpha}}\mu(dt).
\end{align}

%For $\alpha \in (1,2]$, the above formula still holds with the replacement $\rho \mapsto \frac{[(\rho-1)(2-\alpha)+1]\pi}{\alpha}$.  

\begin{rem}\label{Cauchy mixtures} 
Note that the set $\mathcal{B}_{1,\rho}$ coincides with the scale mixtures of the Cauchy distribution $\mathbf{c}_\rho$ with Cauchy transform $G_{\mathbf{c}_\rho}(z)=\frac{1}{z+e^{i\rho\pi}}$.    Since $\mathbf{b}_{1,1}=\delta_1$, 
the set $\mathcal{B}_{1,1}$ coincides with $\mathcal{P}_+$. %We take the convention that $\mathbf{c}_1=\delta_1$ and $\mathbf{c}_{0}=\delta_{-1}$.
\end{rem}

A key for proving the results in this section is the following formulas for the different transforms of scale mixtures of Boolean stable laws.

\begin{prop}\label{prop31} For any admissible pair $(\alpha,\rho)$ and $\mu \in \mathcal{P}_+$, the following formulas hold. 
\begin{align} 
&\label{cauchy1} G_{\mu^{1/\alpha} \circledast\mathbf{b}_{\alpha,\rho}}(z) = -\frac{1}{z}(e^{-i\rho\pi}z)^{\alpha}G_\mu(-(e^{-i\rho\pi}z)^\alpha), &z\in\comp^+,\\ 
%\item $K_{\mu^{1/\alpha}\circledast\mathbf{b}_{\alpha,\rho}}(z) = -e^{i\rho\pi}(e^{-i\rho\pi}z)^{1-\alpha}K_\mu(-(e^{-i\rho\pi}z)^\alpha)$, $z\in \comp^+$; 
&\label{id3} \eta_{\mu^{1/\alpha}\circledast\mathbf{b}_{\alpha,\rho}}(z) = \eta_\mu(-(e^{i\rho\pi}z)^\alpha), &z\in\comp^-.   
\end{align}
\end{prop}
\begin{proof}
Let $X, B_{\alpha,\rho}$ be classical independent random variables following the laws $\mu, \mathbf{b}_{\alpha,\rho}$ respectively. Then 
\begin{equation}\label{eq802}
\begin{split}
G_{\mu\circledast\mathbf{b}_{\alpha,\rho}}(z)
&=G_{X B_{\alpha,\rho}}(z)=\E\!\left[\frac{1}{z-X B_{\alpha,\rho}}\right] = \E\!\left[\frac{1/X}{z/X-B_{\alpha,\rho}}\right] \\
&= \E\!\left[\frac{1}{X}G_{B_{\alpha,\rho}}\left(\frac{z}{X}\right)\right]= \E\!\left[\frac{1}{z+e^{i\alpha\rho \pi} X^\alpha z^{1-\alpha}}\right]
\\
&= z^{\alpha-1}\E\!\left[\frac{1}{z^\alpha+e^{i\alpha \rho\pi} X^\alpha}\right]
= -e^{-i\alpha\rho\pi}z^{\alpha-1}\E\!\left[\frac{1}{-e^{-i\alpha\rho\pi}z^\alpha- X^\alpha}\right]\\
&= -e^{-i\rho\pi}(e^{-i\rho\pi}z)^{\alpha-1}G_{\mu^\alpha}(-(e^{-i\rho\pi}z)^\alpha).
\end{split}
\end{equation}
 By replacing $\mu$ by $\mu^{1/\alpha}$, we obtain (\ref{cauchy1}). The equality (\ref{id3}) follows from (\ref{cauchy1}) and (\ref{etaf}).  
\end{proof}

In particular, for $\rho=1$, we have explicit formulas for the Cauchy transform and related transforms of $\mu\circledast \mathbf{b}_{\alpha,1}$.

\begin{cor} \label{cor03} For $\alpha \in(0,1], \mu \in \mathcal{P}_+$, the following formulas hold. 
\begin{align}
&\label{cauchy2} G_{\mu^{1/\alpha}\circledast\mathbf{b}_{\alpha,1}}(z) = (-z)^{\alpha-1}G_\mu(-(-z)^\alpha), &z<0, \\ 
%\item \label{energy1} $K_{\mu^{1/\alpha}\circledast\mathbf{b}_{\alpha,1}}(z) = (-z)^{1-\alpha}K_\mu(-(-z)^\alpha)$, $z<0$; 
&\label{id4} \eta_{\mu^{1/\alpha}\circledast\mathbf{b}_{\alpha,1}}(z) = \eta_\mu(-(-z)^\alpha), &z<0.   
\end{align}
\end{cor}

%\begin{rem}\label{negative}
%If $\lambda \in\mathcal{P}_-$, then $G_\lambda(z)=-G_{|\lambda|}(-z)$ for $z>0$. Hence, negative versions of Corollary \ref{cor03} can be stated as follows: 
%$G_{\mu^{1/\alpha}\circledast\mathbf{b}_{\alpha,0}}(z) = -z^{\alpha-1}G_\mu(-z^\alpha)$; 
%$\eta_{\mu^{1/\alpha}\circledast\mathbf{b}_{\alpha,0}}(z) = \eta_\mu(-z^\alpha)$ for $z >0$. 
%\end{rem}

Now we show an important formula saying that a scale mixture of $\mathbf{b}_{\alpha,\rho}$ is also a free mixture, and moreover is a monotone mixture.  

\begin{thm}\label{thmconvolutions} For any $\mu,\nu\in\mathcal{P}_+$, the following relations hold: 
\begin{align} 
&\label{c-f1} \mu^{1/\alpha}\circledast\mathbf{b}_{\alpha,\rho} =\mu^{\boxtimes 1/\alpha}\boxtimes\mathbf{b}_{\alpha,\rho}, &\alpha\in(0,1]&, \rho \in \{0,1/2,1\}; \\
&\label{c-f11} \mu^{1/\alpha}\circledast\mathbf{b}_{\alpha,\rho} =\mu \circlearrowright  \mathbf{b}_{\alpha,\rho}, &(\alpha,\rho)\in\mathfrak{A}&. 
\end{align}
\end{thm}
\begin{rem}
The identity (\ref{c-f1}) is valid for $\alpha>1, \rho=1/2$ if $\mu^{\boxtimes 1/\alpha}$ exists in $\mathcal{P}_+$. 
\end{rem}
\begin{proof}
We first show (\ref{c-f1}). (\ref{id3}) implies 
$
\eta_{\mu^{1/\alpha}\circledast\mathbf{b}_{\alpha,\rho}}^{-1}(z)=e^{-i\rho\pi}(-\eta_\mu^{-1}(z))^{1/\alpha}, 
$
so that 
\[
\begin{split}
\Sigma_{\mu^{1/\alpha}\circledast\mathbf{b}_{\alpha,\rho}}(z)
&= \frac{e^{-i\rho\pi}(-\eta_\mu^{-1}(z))^{1/\alpha}}{z}=(-e^{-i\rho\pi})\frac{(-z\Sigma_\mu(z))^{1/\alpha}}{-z} \\
&= \left(\Sigma_{\mu}(z)\right)^{1/\alpha} (-e^{-i\rho\pi})(-z)^{\frac{1-\alpha}{\alpha}}=  \Sigma_{\mu^{\boxtimes 1/\alpha}}(z) \Sigma_{\mathbf{b}_{\alpha,\rho}}(z) 
\end{split}
\]
for $z\in(-c,0)$ where $\Sigma_\mu(z)$ is defined. 
In the last equality, the formula  (\ref{sigmaboole}) was used. 
(\ref{c-f11}) follows from (\ref{id3}) and (\ref{etaboole}). 
\end{proof}

%\begin{rem}
%It follows that $\mathcal{B}_{\alpha,1} \subset \id(\boxplus)$ for $\alpha \in(0,\frac{1}{2}]$ from (1) since $$
%\end{rem}

\begin{cor}
For any probability measures $\mu,\nu\in\mathcal{P}_+$, the following relations hold: 
\begin{align} 
&\label{c-f2} (\mu^{1/\alpha}\circledast\mathbf{b}_{\alpha,1}) \boxtimes (\nu^{1/\alpha}\circledast\mathbf{b}_{\alpha,1}) =\left(\mu\boxtimes\nu \boxtimes\mathbf{b}_{\frac{1}{2-\alpha},1}\right)^{1/\alpha}\circledast\mathbf{b}_{\alpha,1}, &\alpha\in(0,1]&; \\
& \label{c-b} (\mu^{1/\alpha}\circledast\mathbf{b}_{\alpha,\rho}) \uplus (\nu^{1/\alpha}\circledast\mathbf{b}_{\alpha,1}) = (\mu\uplus\nu)^{1/\alpha}\circledast \mathbf{b}_{\alpha,\rho}, &(\alpha,\rho)\in\mathfrak{A}&;\\
&\label{c-m} \mu  \circlearrowright  (\nu^{1/\alpha}\circledast \mathbf{b}_{\alpha,\rho}) = (\mu \circlearrowright \nu)^{1/\alpha}\circledast\mathbf{b}_{\alpha,\rho},&(\alpha,\rho)\in\mathfrak{A}&. 
\end{align}
\end{cor}
\begin{proof}
(\ref{c-f2}) follows from (\ref{c-f1}) and the relation $\mathbf{b}_{\alpha,1} = \left(\mathbf{b}_{\frac{1}{2-\alpha},1}\right)^{\boxtimes 1/\alpha}$. 

(\ref{c-b}) follows from (\ref{id3}). 

(\ref{c-m}) follows from the computation
$$
\eta_{\mu  \circlearrowright  (\nu^{1/\alpha}\circledast\mathbf{b}_{\alpha,\rho})}(z)= \eta_\mu(\eta_{\nu^{1/\alpha}\circledast\mathbf{b}_{\alpha,\rho}}(z)) = \eta_\mu(\eta_\nu(-(e^{i\rho\pi}z)^\alpha)). 
$$
\end{proof}

A particular case of Proposition \ref{prop31} yields a relation between Boolean stable laws with different parameters. 

\begin{prop}\label{prop01}
 The following relation holds for Boolean stable laws: 
 $$
(\mathbf{b}_{\beta,1})^{1/\alpha}\circledast\mathbf{b}_{\alpha,\rho} =\mathbf{b}_{\alpha\beta,\rho}, \qquad (\alpha,\rho)\in\mathfrak{A}, \beta\in(0,1]. 
 $$
\end{prop}
\begin{proof} This is an easy comparison of $\eta$-transforms: 
\begin{eqnarray*}
\eta_{(\mathbf{b}_{\beta,1})^{1/\alpha}\circledast\mathbf{b}_{\alpha,\rho}}(z)&=& \eta_{\mathbf{b}_{\beta,1}}(-(e^{i\rho\pi}z)^\alpha)= -(e^{i\rho\pi}z)^{\alpha\beta} = \eta_{\mathbf{b}_{\alpha\beta,\rho}}(z), 
\end{eqnarray*}
where we used (\ref{id3}) on the first equality and $\eta_{\mathbf{b}_{\beta,1}}(z)=-(-z)^\beta$ on the second. 
\end{proof}

From the previous theorems we can derive closure properties of Boolean mixtures.

\begin{prop} \label{properties}
\begin{enumerate}[\rm(1)]
\item\label{closed}  For $(\alpha,\rho)\in\mathfrak{A}$, the set  $\mathcal{B}_{\alpha,\rho}$ is closed with respect to $\uplus$. 

\item\label{closed10} Let $(\alpha,\rho)\in\mathfrak{A}$. If $\sigma \in \mathcal{P}_+$ and $\tau \in \mathcal{B}_{\alpha,\rho}$, then $ \sigma \circledast\tau,\sigma \circlearrowright\tau\in\mathcal{B}_{\alpha,\rho}$. 

\item\label{closed2} Let $\alpha\leq1$. The set $\mathcal{B}_{\alpha,1}$ is closed with respect to $\circledast, \boxtimes, \uplus, \circlearrowright.$

\item\label{power} Let $\alpha\leq1$. If $\beta\in[\alpha,1]$ and $\mu \in \mathcal{B}_{\alpha,1}$, then $\mu^{\beta} \in \mathcal{B}_{\alpha/\beta,1}.$ 

\item\label{free power} Let $\alpha \in (0,1]$ and $t \geq 1.$ If $\mu \in \mathcal{B}_{\alpha,1}$, then $\mu^{\boxtimes t} \in \mathcal{B}_{\frac{\alpha}{t(1-\alpha)+\alpha},1}$. 

\item\label{inverse} Let $\alpha\leq1$. If $\tau \in \mathcal{B}_{\alpha,1}$ and $\tau(\{0\})=0$, then $\tau^{-1} \in \mathcal{B}_{\alpha,1}.$
\end{enumerate}
\end{prop}
\begin{proof}
(\ref{closed}) follows from (\ref{c-b}). 

(\ref{closed10})\,\, The assertion for $\circledast$ follows by definition. The assertion for $\circlearrowright$ follows from (\ref{c-m}). 

(\ref{closed2})\,\,  The assertions for $\uplus, \circlearrowright, \circledast$ are included in (\ref{closed}) and (\ref{closed10}). The assertion for $\boxtimes$ follows from (\ref{c-f2}). 

%(\ref{cp})\,\, From Theorem \ref{thmconvolutions}(1) and  Proposition \ref{free compound beta}, the measure $\mathbf{b}_{\alpha,1} \circledast\mu^{1/\alpha}$ has the representation  $\MP \boxtimes (\MP^{\frac{1-2\alpha}{\alpha}}\boxtimes \mathbf{f}_\alpha^+ \boxtimes \mu^{\boxtimes 1/\alpha})$. This representation implies the conclusion

%(\ref{inclusion})\,\, It follows from Theorem \ref{thmconvolutions}(\ref{c-b}) that 
%$(\mathbf{b}_{\alpha,1} \circledast\mu^{1/\alpha})^{\uplus t} = \mathbf{b}_{\alpha,1} \circledast (\mu^{\uplus t})^{1/\alpha} \in \mathcal{UI}$ for any $t>0$, so that $\phi(\mathbf{b}_{\alpha,1} \circledast\mu^{1/\alpha})=\infty$ for $\alpha \in (0,\frac{1}{2}]$ and any $\nu\in\mathcal{P}_+$, by Theorem \ref{Bozejko}  

(\ref{power})\,\, From Proposition \ref{prop01} we have $\mathbf{b}_{\beta,1} \circledast(\mathbf{b}_{\alpha/\beta,1})^{1/\beta}=\mathbf{b}_{\alpha,1}$. Taking $\beta$ powers we get  $(\mathbf{b}_{\beta,1})^\beta \circledast\mathbf{b}_{\alpha/\beta,1}=(\mathbf{b}_{\alpha,1})^\beta$, implying that $(\mathbf{b}_{\alpha,1})^\beta\in\mathcal{B}_{\alpha/\beta,1}$. 

(\ref{free power})\,\, 
Take $\mu \in \mathcal{B}_{\alpha,1}$, then $\mu$ is of the form $\nu^{\boxtimes 1/\alpha} \boxtimes\mathbf{b}_{\alpha,1}$, so that 
$$
\mu^{\boxtimes t} =(\nu^{\boxtimes 1/\alpha})^{\boxtimes t}  \boxtimes(\mathbf{b}_{\alpha,1})^{\boxtimes t}=  \nu^{\boxtimes t/\alpha} \boxtimes\mathbf{b}_{\frac{\alpha}{t(1-\alpha)+\alpha},1}, 
$$
where we used (\ref{free power1}) on the last equality. 
We define $\lambda = \nu^{\boxtimes\frac{t}{ t(1-\alpha)+\alpha}}$, to obtain $\mu^{\boxtimes t}=\lambda^{\boxtimes \frac{t(1-\alpha)+\alpha}{\alpha}}\boxtimes  \mathbf{b}_{\frac{\alpha}{t(1-\alpha)+\alpha},1} \in \mathcal{B}_{\frac{\alpha}{t(1-\alpha)+\alpha},1}$. Note that $\frac{t}{ t(1-\alpha)+\alpha} = \frac{t}{t-(t-1)\alpha} \geq 1$, and so $\lambda$ exists as a probability measure.

(\ref{inverse})\,\, This follows from the fact that $(\mathbf{b}_{\alpha,1})^{-1}=\mathbf{b}_{\alpha,1}$ since $\mathbf{b}_{\alpha,1}$ is the law of the quotient of two classical independent, identically distributed positive stable random variables (see (\ref{eq})).    
\end{proof}

We study the behavior of the probability density function at $x=0$. 

\begin{prop}\label{at 0} 
Let $(\alpha,\rho)$ be admissible. If $\tau=\mu\circledast\mathbf{b}_{\alpha,\rho} \in \mathcal{B}_{\alpha,\rho}$ and $\tau\neq \delta_0$, then the density function $p_{\mu,\alpha,\rho}(x)$ of absolutely continuous part of $\tau$ satisfies
\begin{align}
&\liminf_{x\downarrow0}\frac{p_{\mu,\alpha,\rho}(x)}{x^{\alpha-1}}\in(0,\infty],\qquad \text{if~} \rho\neq0,0<\alpha \leq1 \text{~or~} \rho\neq 1/\alpha, 1<\alpha <2, \label{x=+0} \\
&\liminf_{x\uparrow0}\frac{p_{\mu,\alpha,\rho}(x)}{x^{\alpha-1}}\in(0,\infty],\qquad \text{if~}\rho\neq1, 0<\alpha \leq 1\text{~or~} \rho\neq 1-1/\alpha, 1<\alpha<2.  
\end{align}
In particular, $\mathbf{b}_{\beta,\rho} \notin \mathcal{B}_{\alpha,\rho}$ if $0<\alpha <\beta \leq \min(1/\rho,1/(1-\rho))$. 
\end{prop}
\begin{proof}
We can find an interval $[a,b]$ of $(0,\infty)$ such that $\mu([a,b]) >0$. Let $\rho\neq0,\alpha\leq1$. Then for $x>0$, we get 
\[
\begin{split}
p_{\mu,\alpha,\rho}(x)&= (1-\mu(\{0\}))\dfrac{\sin(\pi\rho \alpha)}{\pi} \displaystyle\int_{(0,\infty)}\dfrac{x^{\alpha-1}t^\alpha}{x^{2\alpha}+2(xt)^\alpha \cos(\pi\rho \alpha)+t^{2\alpha}}\mu(dt)\\
&\geq (1-\mu(\{0\}))\dfrac{\sin(\pi\rho \alpha)}{\pi} \displaystyle\int_{[a,b]}\dfrac{x^{\alpha-1}t^\alpha}{x^{2\alpha}+2(xt)^\alpha |\cos(\pi\rho \alpha)|+t^{2\alpha}}\mu(dt)\\
& \geq (1-\mu(\{0\}))\dfrac{\sin(\pi\rho \alpha)}{\pi} \mu([a,b])\dfrac{x^{\alpha-1}a^\alpha}{x^{2\alpha}+2x^\alpha b^\alpha |\cos(\pi\rho \alpha)|+b^{2\alpha}}, 
\end{split}
\]
which leads to the conclusion (\ref{x=+0}). The other cases can be treated similarly.  

If $0<\alpha <\beta$, then for $\mathbf{b}_{\beta,\rho}$ we have $\lim_{x\downarrow0}\frac{p_{\beta,\rho}(x)}{x^{\alpha-1}}=0$. Hence $\mathbf{b}_{\beta,\rho}\notin \mathcal{B}_{\alpha,\rho}$. 
\end{proof}
\begin{prop} Let $\rho\in[0,1]$. Then $\mathcal{B}_{\alpha,\rho}\subset \mathcal{B}_{\beta_,\rho}$ if $0< \alpha <\beta \leq \min(1/\rho,1/(1-\rho))$, where we understand that $1/0=\infty$. The inclusion is strict. 
\end{prop}
\begin{proof} The relation $\mathbf{b}_{\alpha,\rho} \circledast(\mathbf{b}_{\beta,1})^{1/\alpha}=\mathbf{b}_{\alpha\beta,\rho}$ in Proposition \ref{prop01} implies this inclusion. The strictness of the inclusions follows from Proposition \ref{at 0}. 
\end{proof}

\subsection{Connections between Classical, Free and Boolean Stable Laws}\label{Connection}
We want to point out some relations between Boolean, free and classical stable laws. 
%Using this relations we can prove  some of the results in \cite{AHb} more directly. 
As noted in the last paragraph of \cite{AHb}, there is an interplay among free, Boolean and classical stable laws. We have the identity 
$
\textbf{f}_{\alpha,1} \boxtimes (\textbf{f}_{\alpha,1})^{-1} = \textbf{n}_{\alpha,1} \circledast (\mathbf{n}_{\alpha,1})^{-1}
$ for $\alpha \in (0,1]$ as proved in Proposition A4.4 of \cite{BP99}. Moreover, this coincides with a Boolean stable law:  
\begin{equation}\label{eq} 
 \mathbf{b}_{\alpha,1} =\mathbf{f}_{\alpha,1} \boxtimes (\mathbf{f}_{\alpha,1})^{-1} = \mathbf{n}_{\alpha,1} \circledast (\mathbf{n}_{\alpha,1})^{-1},\qquad \alpha\in(0,1].  
\end{equation} 
This relation can be generalized as follows. 
\begin{prop}\label{prop990}
The following formulas hold true. 
\begin{align}
&\mathbf{b}_{\alpha,\rho}= \mathbf{f}_{\alpha,\rho} \boxtimes (\mathbf{f}_{\alpha,1})^{-1}, &\alpha\in(0,1]&, \rho \in \{0,1/2,1\}, \label{b-f1} \\
&\mathbf{b}_{\alpha,\rho}=\mathbf{n}_{\alpha,\rho} \circledast (\mathbf{n}_{\alpha,1})^{-1},&\alpha\in(0,1]&, \rho \in[0,1]. \label{b-c1}
\end{align}
\end{prop}
\begin{rem}
These relations do not hold for $\alpha>1$ since $\mathbf{n}_{\alpha,1}$ and $\mathbf{f}_{\alpha,1}$ are not defined.  
\end{rem}
\begin{proof}
(\ref{b-f1}) follows from (\ref{sym1}) and  the fact 
$$
(\mathbf{f}_{\alpha,1})^{-1}=\MP^{\boxtimes  \frac{1-\alpha}{\alpha}}, 
$$
which can be proved from (\ref{inverse S}), (\ref{S free}) and  (\ref{S MP}). 

(\ref{b-c1})\,\, From Proposition \ref{prop01}, we have, on one hand, that 
$$
\mathbf{b}_{\alpha,\rho} = \mathbf{b}_{\alpha,1} \circledast \mathbf{c}_{\rho}. 
$$
On the other hand, from \cite[Theorem 3.3.1]{Z86}, we get 
$$
\mathbf{n}_{\alpha,\rho}= \mathbf{n}_{\alpha,1} \circledast \mathbf{c}_\rho. 
$$
Hence we get (\ref{b-c1}) by multiplying (\ref{eq}) by $\mathbf{c}_\rho$. 
\end{proof}
%A particular case from above yields a relation between Boolean stable laws with different parameters which generalizes the first item of Proposition 

Here we collect some identities and properties for $\mathbf{b}_{\alpha,\rho}, \mathbf{f}_{\alpha,\rho},\mathbf{n}_{\alpha,\rho}$, including known results which may bring some insight into relationship between different kinds of stable law. 
\begin{thm} \label{thm3.3} The following relations hold. 
\begin{align}
&(\mathbf{b}_{\beta,1})^{1/\alpha}\circledast \mathbf{b}_{\alpha,\rho}=\mathbf{b}_{\alpha\beta,\rho}, &(\alpha,\rho)\in \mathfrak{A}&,\beta \in(0,1], \label{c boole}\\
&(\mathbf{b}_{\beta,1})^{\boxtimes 1/\alpha}\boxtimes\mathbf{b}_{\alpha,\rho}=\mathbf{b}_{\alpha\beta,\rho}, &(\alpha,\rho)\in \mathfrak{A}&,  \beta \in(0,1], \rho \in \{0,1/2,1\}, \label{f boole}\\
&\mathbf{b}_{\beta,1}\circlearrowright \mathbf{b}_{\alpha,\rho}=\mathbf{b}_{\alpha\beta,\rho}, &(\alpha,\rho)\in \mathfrak{A}&,\beta \in(0,1], 
\label{m boole}\\
 &(\mathbf{n}_{\beta,1})^{1/\alpha}\circledast \mathbf{n}_{\alpha,\rho}=\mathbf{n}_{\alpha\beta,\rho}, &(\alpha,\rho)\in \mathfrak{A}&,\beta \in(0,1],\label{c classical}\\
 &(\mathbf{f}_{\beta,1})^{\boxtimes 1/\alpha}\boxtimes\mathbf{f}_{\alpha,\rho}=\mathbf{f}_{\alpha\beta,\rho}, &(\alpha,\rho)\in \mathfrak{A}&,  \beta \in(0,1], \rho \in \{0,1/2,1\}. \label{f free}
\end{align}
Moreover, we have the following properties: 
\begin{align}
&\label{h boole}  \eta_{ \mu^{1/\alpha}\circledast \mathbf{b}_{\alpha,\rho}}(z)= \eta_\mu(-(e^{i\rho\pi}z)^\alpha),  &z\in \comp^-&,(\alpha,\rho)\in \mathfrak{A}, \\
&\label{h classical}  \M_{ \mu^{1/\alpha}\circledast \mathbf{n}_{\alpha,\rho}}(z)= \M_\mu(-(e^{i\rho\pi}z)^\alpha ), & z\in i(-\infty,0)&, (\alpha,\rho)\in \mathfrak{A}, \\
&\label{cum free} \mathcal{C}^\boxplus_{ \mu^{\boxtimes 1/\alpha}\boxtimes \mathbf{f}_{\alpha,\rho}}(z)= \mathcal{C}_\mu^\boxplus(-(e^{i\rho\pi}z)^\alpha ), &z\in (\Gamma_{a,b})^{-1}&, \alpha\leq1, \rho \in \{0,1/2,1\}
\end{align}
for some $a,b>0$ depending on $\mu, \alpha, \rho$. 
 In particular, the maps $\mathbf{B}_{\alpha,\rho},\mathbf{N}_{\alpha,\rho},\mathbf{F}_{\alpha,\rho}: \mathcal{P}_+\to\mathcal{P}$ defined by 
\begin{align*}
&\mathbf{B}_{\alpha,\rho}(\mu) = \mu^{1/\alpha}\circledast\mathbf{b}_{\alpha,\rho}, &(\alpha,\rho)\in \mathfrak{A}&,  \\ 
&\mathbf{N}_{\alpha,\rho}(\mu) = \mu^{1/\alpha}\circledast \mathbf{n}_{\alpha,\rho}, &(\alpha,\rho)\in \mathfrak{A}&, \\
&\mathbf{F}_{\alpha,\rho}(\mu) = \mu^{\boxtimes 1/\alpha}\boxtimes \mathbf{f}_{\alpha,\rho}, &\alpha\in(0,1]&, \rho\in\{0,1/2,1\}
\end{align*}
are homomorphisms with respect to $\uplus, \ast, \boxplus$, respectively. 
\end{thm} 
\begin{rem} We can understand that for $\alpha \leq1$, the formulas (\ref{c boole}) and (\ref{f boole}) are consequences of the formulas (\ref{c classical}) and (\ref{f free}) respectively, together with the identities in Proposition \ref{prop990}. This argument is not available for $\alpha>1$ since Proposition \ref{prop990} is no longer true. 
\end{rem}
\begin{proof} 
(\ref{c boole}) was proved in Proposition \ref{prop01}. (\ref{f boole}) and (\ref{m boole}) follow from (\ref{c boole}), (\ref{c-f1}) and (\ref{c-f11}). 
(\ref{c classical}) is known; see \cite[Theorem 3.3.1]{Z86}. (\ref{f free}) is proved by the direct computation of the $S$-transform (\ref{S free}). 

The formula (\ref{h boole}) is exactly (\ref{id3}). For the formula (\ref{h classical}), let $X, N_{\alpha,\rho}$ be independent random variables following the laws $\mu,\mathbf{n}_{\alpha,\rho}$ respectively. By using (\ref{classical stable}), we have the formula 
\[
\begin{split}
\M_{ \mu^{1/\alpha}\circledast \mathbf{n}_{\alpha,\rho}}(z)
&= \E[e^{z X^{1/\alpha} N_{\alpha,\rho}}] = \E[\exp\left(-e^{i\rho\alpha\pi}(z X^{1/\alpha})^\alpha\right)] \\
&= \E[\exp\left(-X (e^{i\rho\pi}z)^\alpha \right)] = \M_\mu(-(e^{i\rho\pi}z)^\alpha ),\qquad z\in i(-\infty,0). 
\end{split}
\] 

For the formula (\ref{cum free}), we compute 
$$
S_{\mu^{\boxtimes 1/\alpha}\boxtimes \mathbf{f}_{\alpha,\rho}}(z)= -e^{-i\rho\pi}(-z)^{\frac{1-\alpha}{\alpha}}S_\mu(z)^{1/\alpha},  
$$
and hence 
$$
z S_{\mu^{\boxtimes 1/\alpha}\boxtimes \mathbf{f}_{\alpha,\rho}}(z)= e^{-i\rho\pi}(-z S_\mu(z))^{1/\alpha}.   
$$
Due to \cite{APA,NiSp97}, the relation $\mathcal{C}_\nu^\boxplus(z S_\nu(z))=z$ holds for $\nu \in \mathcal{P}_+$ or $\nu\in\mathcal{P}_s$ in an open neighborhood $U$ of $(-a,0)$ for some $a>0$. Therefore $f(z)=z S_\nu(z)$ is univalent in $U$ and $\mathcal{C}_\nu^{\boxplus}$ is univalent in $f(U)$ which contains an interval $(-b,0)$ if $\nu\in\mathcal{P}_+$ and an interval $i(0, c)$ if $\nu\in\mathcal{P}_s$.  
Hence we have $z S_\nu(z)=(\mathcal{C}^\boxplus_\nu)^{-1}(z)$ and then 
$$
(\mathcal{C}^\boxplus_{\mu^{\boxtimes 1/\alpha}\boxtimes \mathbf{f}_{\alpha,\rho}})^{-1}(z) = e^{-i\rho\pi} \left(-(\mathcal{C}^\boxplus_\mu)^{-1}(z)\right)^{1/\alpha},\qquad z \in (-s,0)
$$
for some $s>0$. The formula (\ref{cum free}) follows after some computation and by analytic continuation. 
\end{proof}

As a final comment regarding multiplicative properties of stable laws, we want to point out that the formulas (\ref{f boole}) and (\ref{f free}) are relatives of the reproducing properties 
\begin{align}
&\mathbf{f}_{1/(1+t),1}\boxtimes\mathbf{f}_{1/(1+s),\rho}=\mathbf{f}_{1/(1+s+t),\rho},\label{reproducing free}\\
& \mathbf{b}_{1/(1+t),1}\boxtimes\mathbf{b}_{1/(1+s),\rho}=\mathbf{b}_{1/(1+s+t),\rho} \label{reproducing boole}
\end{align}
for $s,t\geq0$, $(\alpha,\rho)\in\mathfrak{A}$, $\rho\in\{0,1/2,1\}$. The formula (\ref{reproducing free}) was established in \cite{BP99} for $\rho=1$ and in \cite{APA} for $\rho=1/2$, and the formula (\ref{reproducing boole}) was established in \cite{AHb}. We expect these formulas, as well as (\ref{sym1}), (\ref{c-f1}), (\ref{b-f1}), (\ref{f boole}) and (\ref{f free}),  to be true for general $\rho$, but the $S$-transform is not yet available in the general case.

\subsection{Classical and Multiplicative Infinite Divisibility}
We prove the ID part of Theorem \ref{main} and the following paragraph.  
Before proving it, let us recall some facts about exponential mixtures. See \cite{STVH} for further details. 

\begin{defi} A measure is said to be an \textbf{exponential mixture} if $\mu$ is distributed as the random variable $X E$, where $E$ follows the exponential distribution with density $e^{-x}1_{(0,\infty)}(x)$ and $X$ is any random variable independent of $X$. If $X$ is positive then $\mu$ is called a \textbf{positive exponential mixture}. We denote by $\EM$ the set of exponential mixtures. 
\end{defi}
Some properties of exponential functions are the following. 

\begin{enumerate}
\item A positive random variable $X$ is an exponential mixture if and only if $X$ has a completely monotone density. 
\item If $X$ is a positive exponential mixture then $X^\alpha$ is also for $\alpha\geq1$.
\item If $X\in\EM$  and $Y$ is independent of $X$, then $XY\in\EM$.
\end{enumerate}

The importance of exponential mixtures in this paper comes from the following theorem.
\begin{thm}
$\EM \subset \id(\ast)$. 
\end{thm}
Now we are ready to prove part of Theorem \ref{main}. 
\begin{thm}\label{classical ID} 
\begin{enumerate}[\rm(1)]
\item\label{ci} If $\alpha \in (0,1/2], \rho \in[0,1]$, then $\mathcal{B}_{\alpha,\rho}\subset \EM$.  
\item\label{cii} If $\alpha\in (0,1], \rho=1/2$, then $\mathcal{B}_{\alpha,\rho}\subset \id(\ast)$. 
\item\label{ciii} If $\rho \neq 1/2$, then $\mathcal{B}_{1,\rho} \not\subset\id(\ast)$. 
\end{enumerate}
\end{thm}
\begin{proof}
(\ref{ci})\,\, Clearly it is enough to show that $\mathbf{b}_{\alpha,\rho}$ itself is an exponential mixture for $\alpha \leq1/2$. It is proved in \cite{AHb} that $\mathbf{b}_{\alpha,1}$ is a positive exponential mixture. We now use the identity in Proposition \ref{prop01}: $\mathbf{b}_{\beta,\rho}\circledast (\mathbf{b}_{1/2,1})^{1/\beta}=\mathbf{b}_{\beta/2,\rho}$. Since $(\mathbf{b}_{1/2,1})^{1/\beta}$ is a positive exponential mixture for $\beta \leq1$, we see that $\mathbf{b}_{\alpha,\rho}$ is an exponential mixture for $\alpha \leq 1/2$. 

(\ref{cii})\,\, 
Any mixture of a \textbf{symmetric} Cauchy distribution is ID from Theorem IV.10.5 in Steutel and van Harn \cite{STVH}. 
From Remark \ref{Cauchy mixtures} and Proposition \ref{prop01} we have $\mathbf{b}_{\alpha,1/2}= \mathbf{b}_{\alpha,1}\circledast \mathbf{c}_{1/2}$ and hence $\mu\circledast\mathbf{b}_{\alpha,1/2}$ is also a mixture of the symmetric Cauchy distribution $\mathbf{c}_{1/2}$.

(\ref{ciii})\,\, 
Let $p\in(0,1)$, $\rho \neq1/2$ and consider the law $(p\delta_0+(1-p)\delta_1)\ast\mathbf{b}_{1,\rho}$. Its Fourier transform can be computed as 
$$
\mathcal{F}_{(p\delta_0+(1-p)\delta_1)\ast\mathbf{b}_{1,\rho}}(z) = p+(1-p)e^{-(\sin\rho\pi) |z| +i(\cos \rho\pi)z}, ~~z\in\real, 
$$
and in particular 
$$
\mathcal{F}_{(p\delta_0+(1-p)\delta_1)\ast\mathbf{b}_{1,\rho}}\left(\frac{\pi}{\cos\rho\pi}\right) = p-(1-p)e^{-\pi\tan\rho\pi}. 
$$
If we take $p:=\frac{1}{1+e^{\pi\tan\rho\pi}}$, then $\mathcal{F}_{\mathbf{b}_{1,\rho} \ast (p\delta_0+(1-p)\delta_1)}\left(\frac{\pi}{\cos\rho\pi}\right)=0$, which implies 
that $(p\delta_0+(1-p)\delta_1)\ast\mathbf{b}_{1,\rho}\notin\id(\ast)$ from Proposition IV.2.4 in \cite{STVH}.
\end{proof}

\begin{prob}
 Determine the possible pairs $(\alpha,\rho)$ completely so that $\mathcal{B}_{\alpha,\rho} \subset \id(\ast)$. 
 \end{prob}

Now, we prove infinite divisibility of Boolean stable laws in the multiplicative case.
\begin{prop}[Multiplicative infinite divisibility] \label{PropMult}
The positive Boolean stable law $\mathbf{b}_{\alpha,1}$ is infinitely divisible  with respect to the convolutions $\circledast,\boxtimes,\circlearrowright$ for any $\alpha\leq1$. 
\end{prop}
\begin{proof}
The stable distribution $\mathbf{n}_{\alpha,1}$ is $\circledast$-infinitely divisible (see \cite[Theorem 3.5.1]{Z86}), and hence so is $(\mathbf{n}_{\alpha,1})^{-1}$. Therefore $\mathbf{b}_{\alpha,1}= (\mathbf{n}_{\alpha,1})^{-1} \circledast\mathbf{n}_{\alpha,1}$ is $\circledast$-infinitely divisible too (see (\ref{eq})).
The $\boxtimes$-infinite divisibility follows from (\ref{free power1}) and the $\circlearrowright$-infinite divisibility follows from (\ref{multi monotone}).
 \end{proof}

\subsection{Free Infinite Divisibility of $\mathcal{B}_{\alpha,\rho}$}
We prove the free part of Theorem \ref{main} and the following paragraph.  
We start from short proofs of the free infinite divisibility of $\mathcal{B}_{\alpha,1}$ and $\mathcal{B}_{\alpha,1/2}$ by using Proposition \ref{Identities}. 

\begin{prop}\label{free compound beta} 
\begin{enumerate}[\rm(1)]
\item For $\alpha \leq 1/2, \rho\in\{0,1/2,1\}$ and $\mu\in \mathcal{P}_+$, the measure $\mu^{1/\alpha}\circledast\mathbf{b}_{\alpha,\rho}$ is a compound free Poisson with rate $1$ and jump distribution 
$\mu^{\boxtimes1/\alpha} \boxtimes\pi^{\boxtimes \frac{1-2\alpha}{\alpha}}\boxtimes\mathbf{f}_{\alpha,\rho}$, and hence $\mu^{1/\alpha}\circledast\mathbf{b}_{\alpha,\rho}$ is FID. 
\item For $\alpha \leq 2/3$ and $\mu\in\mathcal{P}_+$, the probability measure $\mu^{1/\alpha}\circledast\mathbf{b}_{\alpha,1/2}$ is a compound free Poisson with rate 1 and 
jump distribution
$
\mu^{\boxtimes 1/\alpha}\boxtimes\MP^{\boxtimes \frac{2-3\alpha}{2\alpha}}\boxtimes \sym\!\left(\sqrt{\mathbf{f}_{\alpha/2,1}}\right), 
$
and hence it is FID. 
\end{enumerate}
\end{prop}
\begin{proof}
These are obvious from (\ref{c-f1}), (\ref{sym1}), (\ref{sym2}) and Remark \ref{rem free poisson}.  
\end{proof}

The complete determination of the free infinite divisibility of $\mathcal{B}_{\alpha,\rho}$ requires the ideas of \cite{AHb} and \cite{BH}. 
\begin{defi}
A probability measure $\mu$ is said to be in class $\iu$ if $F_\mu^{-1}$, defined in a domain $\Gamma_{\alpha,\beta}$, has an analytic continuation which is univalent  in $\comp^+$. From the Riemann mapping theorem, $\mu\in\iu$ if and only if there exists a domain $\comp^+\subset D \subset \comp$ such that $F_\mu$ extends to an analytic bijection $\tilde{F}_\mu$ from $D$ onto $\comp^+$. 
\end{defi}

The importance of this class is given by the following lemma (implicitly used in \cite{BBLS}).
\begin{lem}\cite{AHa}\label{lem1}
If $\mu \in \iu$ then $\mu$ is FID.
\end{lem}

The following result was shown in \cite[Proposition 2.1]{BH}. 
\begin{lem}\label{lemBH} A probability measure $\mu$ on $\real$ is in $\mathcal{UI}$ if there exists a simple, continuous curve $\gamma =(\gamma(t))_{t\in\real} \subset \comp^-\cup\real$ with the following properties: 
\begin{enumerate}[\rm(A)]
\item\label{d1} $\displaystyle \lim_{t \to \infty}|\gamma(t)|=\lim_{t \to -\infty}|\gamma(t)|=\infty$; 
\item\label{d4} $F_\mu$ extends to an analytic function $\tilde{F}_\mu$ in $D(\gamma)$ which is continuous on $\overline{D(\gamma)}$, where $D(\gamma)$ denotes the simply connected open set containing $\comp^+$ with boundary $\gamma$;
\item\label{d3} $\tilde{F}_\mu(\gamma) \subset \comp^-\cup\real$; 
\item\label{d5} $\tilde{F}_\mu(z) = z+o(z)$ uniformly as $z \to \infty,~ z \in D(\gamma).$ 
\end{enumerate}
\end{lem}

The following result completes the free part of Theorem \ref{main}. 
\begin{thm}\label{FIDB} The following statements hold. 
\begin{enumerate}[\rm(1)] 
\item\label{a} If $\alpha \in (0,1/2]$ and $\rho\in[0,1]$, then $\mathcal{B}_{\alpha,\rho} \subset \mathcal{UI}\subset \id(\boxplus)$. 
\item\label{b} If $\alpha \in (1/2, 2/3]$ and $\rho\in[2-1/\alpha, 1/\alpha-1]$, then $\mathcal{B}_{\alpha,\rho} \subset \id(\boxplus)$. 
\item\label{c} Otherwise, $\mathcal{B}_{\alpha,\rho} \not\subset \id(\boxplus)$.   
\end{enumerate} 
Moreover, if $(\alpha,\rho)$ satisfies the assumptions of $(\ref{a})$ or $(\ref{b})$, then any probability measure $\nu \in \mathcal{B}_{\alpha,\rho}$ has free divisibility indicator infinity. 
\end{thm}
\begin{proof} 
Let $B_{\alpha,\rho}$ and $X$ be classically independent random variables following $\mathbf{b}_{\alpha,\rho}$ and a probability measure 
$\mu \in \mathcal{P}_+$, respectively. We may assume that $X$ is discrete and takes only a finitely many number of positive values, so in particular $a\leq X \leq b$ for some $0<a<b$. The general case follows from approximation since the set $\iu$ (resp.\ $\id(\boxplus)$) is closed with respect to the weak convergence \cite{AHa} (resp.\ see e.g.\ \cite[Theorem 3.8]{BNT02} and \cite[Lemma 7.8]{Sa99}).

Using (\ref{eq802}), we have
\begin{equation}\label{eq01}
\begin{split}
G_{X B_{\alpha,\rho}}(z)=\E\!\left[\frac{1}{z+e^{i\alpha\rho \pi} X^\alpha z^{1-\alpha}}\right],\qquad z\in\comp^+. 
\end{split}
\end{equation}
We define
\[
\begin{split}
&\theta_{\alpha,\rho} := -\frac{\alpha \rho \pi}{1-\alpha},~~~\phi_{\alpha,\rho} := \frac{(1-\rho\alpha)\pi}{1-\alpha}, \\ 
&\ell_\theta := \{re^{i\theta}: r> 0\},\qquad \theta\in \real, \\
&\gamma:=\ell_{\theta_{\alpha,\rho}}\cup\{0\}\cup \ell_{\phi_{\alpha,\rho}}.  
\end{split}
\]
 % and let $\comp_{A}$ be the sector $\{z\in\comp\setminus\{0\}: \arg z \in A\}$ for $A \subset \real$.

(\ref{a})\,\, 
%Since $B_{\alpha,\rho}$ is a mixture of exponential distributions, so is $B_{\alpha,\rho} X$. 
Assume moreover that $\alpha\in(0,1/2)$; the case $\alpha=1/2$ follows by approximation. It then holds that $\theta_{\alpha,\rho} \in (-\pi, 0]$, $\phi_{\alpha,\rho} \in [\pi, 2\pi)$ and $\phi_{\alpha,\rho} - \theta_{\alpha,\rho} \in(\pi, 2\pi).$ Note that $D(\gamma)=\comp_{(\theta_{\alpha,\rho}, \phi_{\alpha,\rho})}$. 

We will show that the curve $\gamma$ satisfies the assumptions in Lemma \ref{lemBH}. Condition (\ref{d1}) is clear. 
For condition (\ref{d4}), we first show that 
\begin{itemize}
\item[($\ast$)] $G_{X B_{\alpha,\rho}}$ extends analytically to a function $\tilde{G}$ in $D(\gamma)$ which is continuous on $\overline{D(\gamma)}\setminus\{0\}$, and also $\tilde{G}$ does not have a zero in $\overline{D(\gamma)}\setminus\{0\}$. 
\end{itemize}
In view of (\ref{eq01}), it suffices to show that for any $x>0$ and $z\in \overline{D(\gamma)}\setminus\{0\}$, the point $w(x,z):=z+x^\alpha e^{i\alpha\rho \pi} (z)_{(\theta_{\alpha,\rho},\phi_{\alpha,\rho})}^{1-\alpha}$ is not zero. Indeed, when $z\in\comp^+$, $w(x,z)$ is not zero since $w(x,z)=F_{x B_{\alpha,\rho}}(z)\in\comp^+$. When $z=re^{i\theta}$, $\theta \in[\theta_{\alpha,\rho},0]$, we compute the difference of the arguments of the points $z, x^\alpha e^{i\alpha\rho \pi} (z)_{(\theta_{\alpha,\rho},\phi_{\alpha,\rho})}^{1-\alpha}$: 
\begin{equation}\label{estimateA}
0<(1-\alpha)\theta + \alpha\rho\pi -\theta \leq \frac{\alpha}{1-\alpha} \rho\pi < \pi. 
\end{equation}
This shows, for each $z=re^{i\theta}$, $\theta \in[\theta_{\alpha,\rho},0]$, there exists a line $L_z$ passing 0 such that for any $x>0$ the points $z, x^\alpha e^{i\alpha\rho \pi} (z)_{(\theta_{\alpha,\rho},\phi_{\alpha,\rho})}^{1-\alpha}$ lie in the same open half-plane $H_z$ with boundary $L_z$, and hence $w(x,z)$ lies in $H_z$ too, so $w(x,z)\neq0$. 
When $z=re^{i\theta}, \theta\in[\pi, \phi_{\alpha,\rho}]$ we get similarly 
\begin{equation}\label{estimateB}
0>(1-\alpha)\theta + \alpha\rho\pi -\theta \geq -\frac{\alpha}{1-\alpha} (1-\rho)\pi >-\pi. 
\end{equation}
From a similar reasoning, $w(x,z)\neq0$. Since $w(x,z)$ is continuous with respect to $x$, we get $\inf_{x\in[a,b]}|w(x,z)|>0$. Hence we can define the analytic continuation of $G_{X B_{\alpha,\rho}}$ by 
$$
\tilde{G}(z)=\E\!\left[\frac{1}{z+e^{i\alpha\rho \pi} X^\alpha (z)_{(\theta_{\alpha,\rho},\phi_{\alpha,\rho})}^{1-\alpha}}\right], \qquad z\in D(\gamma). 
$$
This extends continuously to $\overline{D(\gamma)}\setminus\{0\}$. The arguments around (\ref{estimateA}) and (\ref{estimateB}) actually show that $\tilde{G}(z)\neq0$ for $z\in \overline{D(\gamma)}\setminus\{0\}$, because $1/w(x,z)$ lies in the half-plane $(H_z)^{-1}$ for any $x>0$ and so $\tilde{G}(z)= \E[1/w(X,z)]\in (H_z)^{-1}$ too. Thus we established $(\ast)$. 

Let $\tilde{F}(z):=1/\tilde{G}(z)$. Then $\tilde{F}$ is analytic in $D(\gamma)$ and continuous on $\overline{D(\gamma)}\setminus\{0\}$ from $(\ast)$. Moreover, since $X$ takes only finitely many values, it is easy to see that $\lim_{z\to0, z\in D(\gamma)}\tilde{G}(z)=\infty$, and hence $\tilde{F}$ extends to a continuous function on $\overline{D(\gamma)}$.  This is condition (\ref{d4}).

For condition (\ref{d3}), take  $r>0$ and then  
\[
\begin{split}
&\frac{1}{re^{i\theta_{\alpha,\rho}}+e^{i\alpha\rho \pi} X^\alpha (re^{i\theta_{\alpha,\rho}})_{(\theta_{\alpha,\rho},\phi_{\alpha,\rho})}^{1-\alpha}}= \frac{1}{re^{i\theta_{\alpha,\rho}}+X^\alpha r^{1-\alpha}} \in \comp^+\cup(0,\infty),\\
&\frac{1}{re^{i\phi_{\alpha,\rho}}+e^{i\alpha\rho \pi} X^\alpha (re^{i\phi_{\alpha,\rho}})_{(\theta_{\alpha,\rho},\phi_{\alpha,\rho})}^{1-\alpha}}=\frac{1}{re^{i\phi_{\alpha,\rho}}-X^{\alpha}r^{1-\alpha}} \in \comp^+\cup(-\infty,0).  
\end{split}
\]  
We take the expectation and use (\ref{eq01}) to obtain $\tilde{G}(\gamma\setminus\{0\})\subset \comp^+\cup\real\setminus\{0\}$. Recall that $\tilde{F}(0)=0$ and so we have condition (\ref{d3}). 

Finally, since $X$ is bounded, it is easy to show that 
\begin{equation}\label{eq1}
z\left(\tilde{G}(z)-\frac{1}{z}\right) =- \E\!\left[\frac{e^{i\alpha\rho \pi} X^\alpha (z)_{(\theta_{\alpha,\rho},\phi_{\alpha,\rho})}^{-\alpha}}{1+e^{i\alpha\rho \pi} X^\alpha (z)_{(\theta_{\alpha,\rho},\phi_{\alpha,\rho})}^{-\alpha}}\right] =o(1)
\end{equation}
uniformly as $z\to\infty$, $z\in D(\gamma)$.   This shows condition (\ref{d5}). From Lemma \ref{lemBH}, the law of $X B_{\alpha,\rho}$ is in $\iu$.

(\ref{b})\,\, Assume moreover that $\rho\in(2-1/\alpha, 1/\alpha-1)$. Note now that $\theta_{\alpha,\pi}\in(-\pi,0), \phi_{\alpha,\rho}\in(\pi,2\pi)$.  Since now $\phi_{\alpha,\rho}-\theta_{\alpha,\rho}>2\pi$, the sector $\comp_{(\theta_{\alpha,\rho}, \phi_{\alpha,\rho})}$ coincides with $\comp\setminus \{0\}$ as a subset of $\comp$, we have to modify Lemma \ref{lemBH}. We use the Riemannian surface corresponding to the interval $(\theta_{\alpha,\rho}, \phi_{\alpha,\rho})$ of arguments and divide the domain into three parts:  $\comp_{(\theta_{\alpha,\rho},\rho\pi)}$, $ \comp_{(\rho\pi,\phi_{\alpha,\rho})}$ and an open neighborhood of $\ell_{\rho\pi}$. We denote by $\tilde{G}_1,\tilde{G}_2,\tilde{G}_3$ analytic maps in these three domains respectively such that each coincides with $G_{X B_{\alpha,\rho}}$ in the intersection of each domain and $\comp^+$, and we denote by $\tilde{F}_i$ their reciprocals. Note that we can define the analytic continuations $\tilde{G}_1,\tilde{G}_2$ along the same line of the previous case (\ref{a}); the inequalities (\ref{estimateA}), (\ref{estimateB}) are still true thanks to the assumption $\rho\in(2-1/\alpha, 1/\alpha-1)$, and so the functions $w_1(x,z):=z+x^\alpha e^{i\alpha\rho \pi} (z)_{(\theta_{\alpha,\rho},\rho\pi)}^{1-\alpha}$ and $w_2(x,z):=z+x^\alpha e^{i\alpha\rho \pi} (z)_{(\rho\pi,\phi_{\alpha,\rho})}^{1-\alpha}$ do not vanish.  Hence we have the expression for $\tilde{G}_1$ as 
$$
\tilde{G}_1(z)=\E\!\left[\frac{1}{z+e^{i\alpha\rho \pi} X^\alpha (z)_{(\theta_{\alpha,\rho},\rho\pi)}^{1-\alpha}}\right] 
$$
and similarly for $\tilde{G}_2$. The map $\tilde{G}_3$ is just the restriction of $G_{X B_{\alpha,\rho}}$. 
Note that $\tilde{F}_3(\ell_{\rho\pi})=\ell_{\rho\pi}$ and  $\tilde{F}_3$ is univalent in an open neighborhood $D_3$ of $\ell_{\rho\pi}$ from a direct computation of derivative of $\tilde{G}_3$. So the left compositional inverse $(\tilde{F}_3|_{D_3})^{-1}$ exists in an open neighborhood of $\ell_{\rho\pi}$. 

We want to define a univalent inverse of $\tilde{F}_1$ in $\comp_{(0,\rho\pi)}$. Now we take the curve  $\gamma_1=\ell_{\theta_{\alpha,\rho}}\cup \{0\}\cup \ell_{\rho\pi}$ as the curve $\gamma$ in Lemma \ref{lemBH}. We can check the conditions in Lemma \ref{lemBH} similarly to (\ref{a}) except that we understand that $D(\gamma_1)= \comp_{(\theta_{\alpha,\rho},\rho\pi)}$ and we replace condition (\ref{d3}) by $\tilde{F}_1(\gamma_1)\subset (\comp_{(0,\rho\pi)})^c$.  Accordingly to these modifications, the conclusion of the lemma changes to: there is a domain $D_1\subset \comp_{(\theta_{\alpha,\rho},\rho\pi)}$ such that $\tilde{F}_1$ is a bijection from $D_1$ onto $\comp_{(0,\rho\pi)}$. The proof of this fact is almost the same as \cite[Proposition 2.1]{BH}.  Hence its inverse map $(\tilde{F}_1|_{D_1})^{-1}$ exists in $\comp_{(0,\rho\pi)}$. Similarly, the inverse $(\tilde{F}_2|_{D_2})^{-1}$ exists in $\comp_{(\rho\pi,\pi)}$ for some $D_2$. Finally we define an analytic map $\tilde{F}^{-1}$ in $\comp^+$ by 
\[
\tilde{F}^{-1}(z)=
\begin{cases} 
 (\tilde{F}_1|_{D_1})^{-1}(z), &z\in \comp_{(0, \rho\pi)}, \\
 (\tilde{F}_3|_{D_3})^{-1}(z), &z\in \ell_{\rho\pi}, \\ 
  (\tilde{F}_2|_{D_2})^{-1}(z), &z\in \comp_{(\rho\pi,\pi)}. 
\end{cases}
\] 
This map is not necessarily univalent, but we can show that $\tilde{\phi}(z):=\tilde{F}^{-1}(z)-z$ for $z\in\comp^+$ takes values in $\comp^-\cup\real$; see the arguments  in \cite[Lemma 2.7]{H} or in \cite[Proposition 3.6]{AHb}. Since $\tilde{\phi}$ is the analytic continuation of the Voiculescu transform $\phi_{X B_{\alpha,\rho}}$, the law of $X B_{\alpha,\rho}$ is FID from Theorem \ref{thmBV93}.

(\ref{c})\,\, As proved in \cite{AHb}, $\mathbf{b}_{\alpha,\rho}\notin \id(\boxplus)$ in the following cases: $\alpha>1$; $\alpha\in(1/2,1)$ and $\rho \in [0,\frac{2\alpha-1}{\alpha})\cup(\frac{1-\alpha}{\alpha},1]$. The remaining case is $\alpha=1$ when $\mathbf{b}_{\alpha,\rho}$ is a Cauchy distribution, which itself is FID. However we can show $\mathcal{B}_{1,\rho} \not\subset \id(\boxplus)$; see Proposition \ref{exacauchy}.  

For the final statement, take $\nu \in \mathcal{B}_{\alpha,\rho}$, which may be written as $\nu=\mu^{1/\alpha}\circledast\mathbf{b}_{\alpha,\rho}$. For any $t>0$, from (\ref{id3}) we have 
$$
\nu^{\uplus t}= (\mu^{\uplus t})^{1/\alpha}\circledast\mathbf{b}_{\alpha,\rho} \in \mathcal{B}_{\alpha,\rho} \subset \id(\boxplus), 
$$
and hence $\phi(\nu)=\infty$ from (\ref{free div ind char}). 
\end{proof}
%\begin{rem}\begin{enumerate}[\rm(1)] 
%\item For a general $X$, not necessarily nonnegative, the random variable $B_{\alpha,\rho}X$ may be FID, but this problem is more complicated and not so important in the following  analysis. So we only consider nonnegative random variables $X$ here. 
%\end{enumerate} 
%\end{rem} 
%PLEASE CHEK??? 
\begin{rem}
In the context of complex analysis, the map $F_{\mu\circledast\mathbf{b}_{\alpha,\rho}}$ may be useful because it has the invariant half line $\ell_{\rho\pi}=\{re^{i\rho\pi}: r>0\}$. 
%In particular, for $\rho\in(0,1)$, the map $F_{\mathbf{b}_{\alpha,\rho}\circledast \mu}$ is hyperbolic ?? in the sense of???
\end{rem}

For a nonnegative finite measure $\sigma$ on $(0,1/2]$, the continuous Boolean convolution \cite{AHb} is the probability measure  $\mathbf{b}(\sigma)$ defined by 
$$
\eta_{\textbf{b}(\sigma)}(z) = -\int_{(0,1/2]} (-z)^{\alpha}\, \sigma(d\alpha).  
$$ 
We can similarly prove the free infinite divisibility for the scale mixture $\mu\circledast\mathbf{b}(\sigma)$. 
However it turns out that $\mathbf{b}(\sigma)$ belongs to $\mathcal{B}_{1/2,1}$ as we see in Proposition \ref{Boolean integral}.

\section{Examples}\label{Examples}

\subsection{Explicit Densities of Probability Measures in $\mathcal{B}_{1/2,1}$}
The probability density function (\ref{density mixture}) of $\mu \circledast \mathbf{b}_{1/2,1}$ is in particular simply written as
\begin{equation}\label{density 1/2,1}
\frac{x^{-1/2}}{\pi}\int_{0}^{\infty} \frac{\sqrt{y}}{x+y}\, \mu(d y),\qquad x>0. 
\end{equation}
By introducing the measure $\tau(d y)=\sqrt{y}\,\mu(dy)$, the density has the expression 
\begin{equation}\label{stieltjes density}
-\frac{x^{-1/2}}{\pi} G_\tau(-x),\qquad x>0. 
\end{equation}
We will find explicit probability densities of this form. 

\begin{prop}\label{gen beta1}
Let $\alpha \in(1/2, 1]$ and $\alpha\beta\in(0,1]$. The generalized beta distribution of the second kind with density function 
\begin{equation}\label{eq001}
c_{\alpha,\beta}\cdot\frac{x^{\alpha-3/2}}{(x^{\alpha\beta}+1)^{1/\beta}}1_{(0,\infty)}(x) 
\end{equation}
belongs to the class $\mathcal{B}_{1/2,1}$, and hence it is in $\id(\boxplus)\cap \EM$ from Theorems \ref{classical ID} and \ref{FIDB}.  Note that $c_{\alpha,\beta}>0$ is a normalizing constant. 
\end{prop}
\begin{proof} From (\ref{density 1/2,1}) and (\ref{stieltjes density}), it suffices to find a measure $\tau$ such that $y^{-1/2}\tau(dy)$ is a finite measure and that 
$-G_\tau(-x)= \frac{x^{\alpha-1}}{(x^{\alpha\beta}+1)^{1/\beta}}$; then we may define $\mu=c\cdot y^{-1/2}\tau(dy)$ for a normalizing constant $c>0$. 

We define an analytic map
$$
G(z):=-\frac{(-z)^{\alpha-1}}{((-z)^{\alpha\beta}+1)^{1/\beta}},\qquad z\in\comp\setminus\real_+
$$
and we show that this is the Cauchy transform of a probability measure. 
Since $zG(z) =o(1)$ uniformly as $z\to\infty, z\in\comp^+$, it suffices to show that $G$ maps $\comp^+$ into $\comp^-$. 
For $z=r e^{i\theta}$, $r>0,\theta \in (0,\pi)$, we have 
\[
\begin{split}
\im(G(r e^{i\theta}))
&= -\text{Im}\!\left( \frac{(r e^{i(\theta-\pi)})^{\alpha-1}}{((r e^{i(\theta-\pi)})^{\alpha\beta}+1)^{1/\beta}}\right)  
=-r^{\alpha-1} \text{Im}\!\left( \frac{e^{i(1-\alpha)(\pi-\theta)} }{(r^{\alpha\beta}e^{-i\alpha\beta(\pi-\theta)}+1)^{1/\beta}}\right). 
%&= \text{Im}\!\left( \frac{e^{-i\alpha\pi} x^{\alpha-1}}{(e^{-i\alpha\beta\pi}x^{\alpha\beta}+1)^{1/\beta}}\right).  
\end{split}
\]
Let $\varphi(r e^{i\theta}):=\arg(r^{\alpha\beta}e^{-i\alpha\beta(\pi-\theta)}+1)$ and $R(r e^{i\theta}):=|r^{\alpha\beta}e^{-i\alpha\beta(\pi-\theta)}+1|$. Since $\alpha\beta \in(0,1]$, 
it is easy to see that $r^{\alpha\beta}e^{-i\alpha\beta(\pi-\theta)}+1\in\comp^-$ and $\varphi(r e^{i\theta})\in (-\alpha\beta(\pi-\theta),0)$. We have the expression
\begin{equation}\label{x near 0}
\im(G(r e^{i\theta}))=-r^{\alpha-1}R(r e^{i\theta})^{-1/\beta} \sin\left((1-\alpha)(\pi-\theta)-\varphi(r e^{i\theta})/\beta\right). 
\end{equation}
Since $\varphi(r e^{i\theta})\in (-\alpha\beta(\pi-\theta),0)$, we get $0<(1-\alpha)(\pi-\theta)-\varphi(r e^{i\theta})/\beta<\pi-\theta <\pi$, and hence 
$\im(G(r e^{i\theta}))<0$. 

Now we know that there exists $\tau\in\mathcal{P}$ such that $G=G_\tau$. Since $G$ takes real values on $(-\infty,0)$, it follows from the Stieltjes inversion that $\tau\in\mathcal{P}_+$.  Both $\varphi$ and $R$ extend continuously to $\comp^+\cup \real\setminus\{1\}$ ($R$ extends to $\comp^+\cup \real$. $\varphi$ also extends to $\comp^+\cup \real$ if $\alpha\beta<1$). Therefore (\ref{x near 0}) gives us
\begin{equation}\label{x near 01}
\lim_{y\downarrow0}\im(G(x+i y))=-x^{\alpha-1}R(x)^{-1/\beta}\sin\left(\alpha\pi+\varphi(x)/\beta\right), \qquad x>0, x\neq1. 
\end{equation}
By the Stieltjes inversion, $\tau$ has a density which behaves as $\frac{1+o(1)}{\pi}\sin(\alpha\pi) x^{\alpha-1}$ as $x\downarrow0$, and hence $x^{-1/2}\tau(d x)$ is a finite measure for $\alpha \in(1/2,1]$. 
%Let $c>0$ be a normalizing constant such that $\mu=c\cdot x^{-1/2}\tau(d x)$ is a probability measure. 
Thus $\mu\circledast\mathbf{b}_{1/2,1}$ has the density (\ref{eq001}) thanks to the arguments in the first paragraph of the proof. 
\end{proof}

\begin{prop} Let $-1<a< 1/2$. The probability measure  with density
$$
c_a\cdot \frac{(1+x)^a-1}{a x^{3/2}}1_{(0,\infty)}(x)
$$
 belongs to $\mathcal{B}_{1/2,1}$, where $c_a>0$ is a normalizing constant.  If $a=0$, this measure is understood as 
 $$
 c_0\cdot\frac{\log(1+x)}{x^{3/2}}1_{(0,\infty)}(x)\,dx. 
 $$
 \end{prop}
 \begin{proof}
 %$$ appears, which is in $\mathcal{UI}$, but not in $\mathcal{B}_{1/2}^+$. Therefore the set $\mathcal{B}_{1/2}^+$ is not closed with respect to the weak convergence. 
First consider $-1<a<0$ and let $\tau_a$ be the shifted beta distribution of the second kind with density $\frac{1}{B(1+a,-a)}\cdot\frac{(t-1)^a}{t}$ on $(1,\infty)$. 
From Example 3.3(4) in \cite{H}, we get $-G_{\tau_a}(-x)=\frac{1-(1+x)^a}{x}$, which can be written as 
\begin{equation}\label{eq891}
\int_1^\infty \frac{1}{x+t} \cdot \frac{(t-1)^a}{t}\,dt= \frac{\pi a}{\sin(\pi a)}\cdot \frac{(1+x)^a-1}{a x},\qquad a \in(-1,0). 
\end{equation}
This identity extends to $a\in(-1,1)$ since the integral in the LHS exists and the both hands sides are real analytic functions of $a\in(-1,1)$. 
Let  $\mu_a$ be the probability measure on $(1,\infty)$ with density $\frac{1}{B(1+a,1/2-a)}\cdot\frac{(t-1)^a}{t^{3/2}}$ for $a\in(-1,1/2)$. From (\ref{stieltjes density}) and (\ref{eq891}), up to the multiplication of a constant the measure $\mu_a \circledast\mathbf{b}_{1/2,1}$ has the density 
$ \frac{(1+x)^{a}-1}{a x^{3/2}}.$    
\end{proof}

We present the third example without a proof. 
\begin{exa}
Let $\mu$ be the beta distribution with density $\frac{1}{2\sqrt{t}}1_{(0,1)}(t)\,dt$. Then the measure $\mu\circledast\mathbf{b}_{1/2,1}$ is given by 
$$
\frac{\log\left(1+1/x\right)}{2\pi\sqrt{x}}1_{(0,\infty)}(x)\,dx. 
$$
\end{exa}

\subsection{Limit Distributions of Multiplicative Free Laws of Large Numbers}
Tucci investigated free multiplicative laws of large numbers for measures with compact support in \cite{T10}
and then Haagerup and M\"{o}ller proved the general case as follows \cite{HM}: 
If $\mu\in\mathcal{P}_+$, then the law
$$
(\mu ^{\boxtimes n})^{1/n}
$$
weakly converges to a probability measure on $\real_+$, which we denote by $\Phi(\mu).$ A striking fact is that the limit law $\Phi(\mu)$ is not a delta measure unless $\mu$ is a delta measure. In fact the map $\Phi$ is even injective. The distribution function of this limit measure can be described in terms of 
the $S$-transform as follows: 
$$
\Phi(\mu)(\{0\})=\mu(\{0\}), \qquad\Phi(\mu)\left(\left[0, \frac{1}{S_\mu(x-1)}\right]\right)=x,\qquad x\in(\mu(\{0\}),1).  
$$
We compute $\Phi(\mu)$ when $\mu$ is a scale mixture of positive Boolean stable laws. 
%It follows then from Proposition \ref{properties}(\ref{free power})  that the measure $(\mu ^{\boxtimes n})^{1/n}$ belongs to class $\mathcal{B}_{\frac{n\alpha}{n(1-\alpha)+\alpha},1}$. Note that $\frac{n\alpha}{n(1-\alpha)+\alpha} \in (0,1)$. 
%If moreover we assume that $\alpha\in (0,\frac{1}{3}]$, then 
%$$
%\frac{n\alpha}{n(1-\alpha)+\alpha} \leq \frac{n\alpha}{n(1-\alpha)} = 1/2. 
%$$
%Taking the limit $n \to \infty$, we have the following. 
\begin{thm}\label{large} Let $\mu\in\mathcal{P}_+$ and $\alpha\in(0,1/2].$ 
\begin{enumerate}[\rm(1)]

\item It holds that \begin{equation}\label{eq0989}
\Phi(\mu\boxtimes\mathbf{b}_{\alpha,1})= (\mu\boxtimes\mathbf{b}_{\frac{\alpha}{1-\alpha},1})\circledast\mathbf{Pa}(1), 
\end{equation}
where $\mathbf{Pa}( r)$ is the Pareto distribution 
$$
\mathbf{Pa}( r)(dx)= r(1+x)^{-r-1}\,1_{(0,\infty)}(x)\,dx.  
$$ 
In particular, 
\begin{equation}\label{eq0991}
\Phi(\mu\boxtimes\mathbf{b}_{1/2,1})=\mu\circledast \mathbf{Pa}(1). 
\end{equation}

\item We have 
\begin{equation}\label{eq0990} 
\Phi(\mu^{1/\alpha}\circledast\mathbf{b}_{\alpha,1})= (\mu^{\boxtimes \frac{1}{1-\alpha}})^{\frac{1-\alpha}{\alpha}} \circledast \mathbf{Pa}(1)\circledast\mathbf{b}_{\frac{\alpha}{1-\alpha},1}. 
\end{equation}
This implies that $\Phi(\mathcal{B}_{\alpha,1}) \subset \mathcal{B}_{\frac{\alpha}{1-\alpha},1} \cap \EM$ since $\mathbf{Pa}(1)\in\EM$. In particular, we have
$\Phi(\mathcal{B}_{1/3,1}) \subset \id(\boxplus) \cap \EM.$

\end{enumerate}
\end{thm}
\begin{proof}
First we show (\ref{eq0991}) as follows:  
 \[
 \begin{split}
\left( (\mu\boxtimes\mathbf{b}_{1/2,1})^{\boxtimes n}\right)^{\frac{1}{1+n}} 
&= \left(\mu^{\boxtimes n}\boxtimes(\mathbf{b}_{1/2,1})^{\boxtimes n}\right)^{\frac{1}{1+n}} \\
&= \left(\mu^{\boxtimes \frac{n}{1+n}})^{\boxtimes (1+n)}\boxtimes\mathbf{b}_{\frac{1}{1+n},1}\right)^{\frac{1}{1+n}} \\ 
&= \left((\mu^{\boxtimes \frac{n}{1+n}})^{1+n}\circledast\mathbf{b}_{\frac{1}{1+n},1}\right)^{\frac{1}{1+n}} \\ 
&= \mu^{\boxtimes \frac{n}{1+n}}\circledast\left(\mathbf{b}_{\frac{1}{1+n},1}\right)^{\frac{1}{1+n}},  
 \end{split}
 \]
 where we used (\ref{free power1}) on the second line. 
If a measure $\nu \in \mathcal{P}_+$ has a density $p(x)$, the measure $\nu^q$ has the density $\frac{1}{q}x^{\frac{1}{q}-1}p(x^{\frac{1}{q}})$. It then follows that 
the density of $(\mathbf{b}_{q,1})^{q}$ is given by 
$$
\frac{\sin(\pi q)}{\pi q}\cdot \frac{1}{x^2+2x\cos(\pi q)+1}, 
$$
which converges to $(1+x)^{-2}$ uniformly on $[0,\infty)$ as $q\to0$. Hence (\ref{eq0991}) has been proved. 

Recall from (\ref{reproducing boole}) that the identity $\mathbf{b}_{1/2,1} \boxtimes \mathbf{b}_{\frac{\alpha}{1-\alpha},1} = \mathbf{b}_{\alpha,1}$ holds. By replacing $\mu$ by $\mathbf{b}_{\frac{\alpha}{1-\alpha},1} \boxtimes \mu$ in (\ref{eq0991}), we obtain (\ref{eq0989}). 

By replacing $\mu$ by $\mu^{\boxtimes 1/\alpha}$ in (\ref{eq0989}), we have 
 \[
 \begin{split}
\Phi(\mu^{1/\alpha}\circledast\mathbf{b}_{\alpha,1}) 
&=\Phi(\mu^{\boxtimes 1/\alpha}\boxtimes\mathbf{b}_{\alpha,1}) \\
&= \left( (\mu^{\boxtimes \frac{1}{1-\alpha}})^{\boxtimes\frac{1-\alpha}{\alpha}}\boxtimes\mathbf{b}_{\frac{\alpha}{1-\alpha},1}\right)\circledast \mathbf{Pa}(1)\\
&= \left( (\mu^{\boxtimes \frac{1}{1-\alpha}})^{\frac{1-\alpha}{\alpha}}\circledast\mathbf{b}_{\frac{\alpha}{1-\alpha},1}\right)\circledast \mathbf{Pa}(1). 
%\\&= \mathbf{b}_{\frac{\alpha}{1-\alpha},1}\circledast \left(\mathbf{Pa}(1)\circledast (\mu^{\boxtimes \frac{1}{1-\alpha}})^{\frac{1-\alpha}{\alpha}}\right), 
 \end{split}
 \] 
\end{proof}

\begin{exa}\label{gen beta2} Theorem \ref{large} in particular implies that $\Phi(\mathcal{B}_{\alpha,1}) \subset \id(\boxplus)$ for $\alpha \leq1/3$. The constant $1/3$ is optimal as shown in the following example. 
Take $\mu$ to be the Boolean stable law $\mathbf{b}_{\alpha,1}$ itself for $\alpha \in (0,1)$. Then 
$$
S_\mu(z)= \Sigma_\mu\left(\frac{z}{1+z}\right) = \left(\frac{-z}{1+z}\right)^{\frac{1-\alpha}{\alpha}}=\left(\frac{-z}{1+z}\right)^{1/\beta}, 
$$
where the new parameter $\beta= \frac{\alpha}{1-\alpha}\in(0,\infty)$ is introduced for simplicity. The compositional inverse function of $\frac{1}{S_\mu(x-1)}$ is now equal to 
$\frac{x^\beta}{1+x^\beta}$, and so we have 
$\Phi(\mu)([0,x])=\frac{x^\beta}{1+x^\beta}$ for $x\in[0,\infty)$. The density function is given by 
$$
\frac{d\Phi(\mu)}{dx}(x) = \frac{\beta x^{\beta-1}}{(x^\beta+1)^2},\qquad x\in(0,\infty),  
$$
which is a generalized beta distribution of second kind but a different one from Proposition \ref{gen beta1}. This measure is in $\id(\boxplus)\cap\EM$ for $\beta \in(0,1/2]$ from Theorem \ref{large}. From \cite[Theorem 5.1]{H}, the measure $\Phi(\mu)$ is not in $\id(\boxplus)$ for $\beta\in(1/2, 2/3)$, 
and so the number $1/3$ is optimal. 
\end{exa}

\subsection{Continuous Boolean Convolution}
The \textbf{continuous Boolean convolution} $\mathbf{b}(\sigma)$ of Boolean stable laws is defined by 
$$
\eta_{\mathbf{b}(\sigma)}(z)= -\int_{(0,1]} (-z)^{\alpha}\, \sigma(d\alpha),\qquad z\in\comp^-, 
$$
 for nonnegative finite measure $\sigma$ supported on $(0,1]$ (see \cite{AHb}). 
Symbolically this measure may be written as 
$$
\textbf{b}(\sigma) = \int^{\uplus}_{(0,1]} \mathbf{b}_{\alpha,1}\, \sigma(d\alpha). 
$$
The density is given by 
$$
\frac{1}{\pi}\cdot\frac{\int_{(0,1]}\sin(\alpha\pi)x^{1-\alpha}\,\sigma(d\alpha)}{\left(x+\int_{(0,1]}\cos(\alpha\pi)x^{1-\alpha}\,\sigma(d\alpha)\right)^2 +\left(\int_{(0,1]}\sin(\alpha\pi)x^{1-\alpha}\,\sigma(d\alpha)\right)^2},~~~x>0. 
$$

\begin{prop}\label{Boolean integral} For nonnegative finite measure $\sigma$ on $(0,1/2]$, we have 
$$
\mathbf{b}(\sigma)=\mathbf{b}(D_2\sigma)^2\circledast\mathbf{b}_{1/2,1}  \in \mathcal{B}_{1/2,1}. 
$$
\end{prop}

\begin{proof}
We compare the $\eta$-transforms using Corollary \ref{cor03}: 
\begin{eqnarray*}
\eta_{\mathbf{b}(D_2\sigma)^2\circledast\mathbf{b}_{1/2,1} }(z)&=&\eta_{\mathbf{b}(D_2\sigma)}(-(-z)^{1/2})\\
&=&-\int_0^1(-z)^{\frac{1}{2}\alpha}\,D_2\sigma(d\alpha) \\
&=&-\int_0^{1/2}(-z)^{\alpha }\,\sigma(d\alpha)\\
&=&\eta_{\mathbf{b}(\sigma)}(z). 
\end{eqnarray*}
\end{proof}

\begin{exa} A particularly interesting case comes when $\sigma:= \sum_{k=1}^n \binom{n}{k}\delta_{\frac{k}{n}\alpha}$. In this case we get
$$
\mathbf{b}(\sigma)=  (\mathbf{m}_{1/n,1})^{1/\alpha} \circledast \mathbf{b}_{\alpha,1} \in \mathcal{B}_{\alpha,1}. 
$$ 
This can be proved by computing the $\eta$-transform (see (\ref{etamonotone})). 
\end{exa}

\subsection{Probability Measures in $\mathcal{B}_{1,\rho}\setminus\id(\boxplus)$}
We present a two-parameter family of probability measures, some of which belong to $\mathcal{B}_{1,\rho}\setminus\id(\boxplus)$. This completes the proof of Theorem \ref{FIDB}. 
For $t,\rho\in[0,1]$, let $\lambda_ {t,\rho} \in \mathcal{B}_{1,\rho}$ be the probability measure 
$$
\lambda_ {t,\rho}:=((1- t)\delta_0 +  t\delta_1)\circledast \mathbf{c}_{\rho}=  t\delta_0 + \frac{1- t}{\pi}\cdot \frac{\sin\rho\pi}{(x+\cos \rho\pi)^2+\sin ^2\rho\pi}\,1_{\real}(x)\,dx, 
$$
which appeared in the proof of Theorem \ref{classical ID}. 
The measures $\lambda_{0, t}$ and $\lambda_{1, t}$ are understood to be $ t\delta_0+(1- t)\delta_{-1}$ and $ t\delta_0+(1- t)\delta_{1}$ respectively.  

\begin{prop}\label{exacauchy}
The measure $\lambda_ {t,\rho}$ is FID if and only if: $(1)$ $ t=0$; $(2)$ $ t \in[1/2,1],~ |\cos \rho\pi| \leq 2 t-1$. 
In particular, $\mathcal{B}_{1,\rho} \not\subset \id(\boxplus)$. 
\end{prop}
%\begin{rem}
%We do not know the complete classification of $\lambda_$
%\end{rem}
\begin{proof}
Assume that $t,\rho\in(0,1)$; the other cases are Bernoulli distributions and are well known. The Cauchy transform of $\lambda_ {t,\rho}$ is given by $G_{\lambda_ {t,\rho}}(z)=\frac{ t}{z}+\frac{1- t}{z+e^{i\rho\pi}}$, and so for some $\alpha,\beta>0$
$$
\phi_{\lambda_ {t,\rho}}(z)=\frac{1}{2}\left(-z-e^{i\rho\pi}+\left(z^2+2(2 t-1)e^{i\rho\pi}z +e^{2i\rho\pi}\right)_{(0,2\pi)}^{1/2}\right),\qquad z\in\Gamma_{\alpha,\beta}. 
$$
Note that the polynomial $z^2+2(2 t-1)e^{i\rho\pi}z +e^{2i\rho\pi}$ has the zeros $z_{\pm}=e^{i(\rho\pi \pm \phi)}$, where $\cos \phi=1-2 t$ and $\phi\in(0,\pi)$. If $ t\in(0,1/2)$, then $z_{+}$ or $z_-$ is contained in $\comp^+$ and so $\phi_{\lambda_ {t,\rho}}$ does not extend to $\comp^+$ analytically, and so $\lambda_ {t,\rho} \notin \id(\boxplus)$ from Theorem \ref{thmBV93}. 
 If $ t\in[1/2,1)$, then both $z_{+}=e^{i(\rho\pi+\phi)}$ and $z_-=e^{i(\rho\pi-\phi)}$ are in $\comp^- \cup \real$ if and only if $\pi-\phi \leq \rho\pi \leq \phi$ or equivalently $|\cos\rho\pi| \leq 2 t-1.$ 
 
 Now suppose that $|\cos\rho\pi| \leq 2 t-1$. Then $\phi_{\lambda_ {t,\rho}}$ extends analytically to $\comp^+$ and continuously to $\comp^+\cup \real$, so from Theorem \ref{thmBV93}
we only have to show that 
\begin{equation}\label{eqphi}
\text{Im}(\phi_{\lambda_ {t,\rho}}(z)) \leq 0,\qquad z\in\comp^+\cup\real. 
\end{equation}
First we are going to prove that $\text{Im}(\phi_{\lambda_ {t,\rho}}(x+i0)) \leq 0$ for $x\in\real$. 
Let 
\[
\begin{split}
re^{i\theta}:&=x^2+2(2 t-1)e^{i\rho\pi}x+e^{2i\rho\pi}\\
&=x^2+2(2 t-1)(\cos\rho\pi)x +\cos2\rho\pi +2\sin\rho\pi(\cos\rho\pi+(2 t-1)x)i.
\end{split}
\]
 The inequality (\ref{eqphi}) on $\real$ is equivalent to $\sqrt{r}\sin(\theta/2) \leq \sin \rho\pi$ and from the formula $\sin^2(\theta/2)=(1-\cos\theta)/2$, it is also equivalent to $r \leq r\cos \theta + 2\sin^2\rho\pi =x^2+2(2 t-1)(\cos\rho\pi)x +1$.  The difference $(x^2+2(2 t-1)(\cos\rho\pi)x +1)^2-r^2$ turns out to be $16 t(1- t)(\sin^2 \rho\pi)x^2 \geq0$, showing the desired conclusion (\ref{eqphi}) for $z\in\real$. 

Next, consider the bounded domain $D_R$ surrounded by the boundary $[-R,R]\cup \{z\in\comp^+: |z|=R\}$. 
We can easily show the estimate $\phi_{\lambda_ {t,\rho}}(z) = -(1- t)e^{i\rho\pi} +o(1)$ uniformly as $z\to \infty$, $z\in\comp^+$. Hence, (\ref{eqphi}) is valid on the boundary of $D_R$ for large $R>0$. From the maximum principle for (sub-) harmonic functions, the inequality (\ref{eqphi}) holds for any $z \in D_R$, and hence for $z\in\comp^+$ by taking the limit $R\to \infty$. 
\end{proof}

Moreover we can explicitly calculate the free divisibility indicator of $\lambda_ {t,\rho}$.

\begin{prop}
Let $ t,\rho\in[0,1]$. Then 
$$
\phi\left(\lambda_{ t,\rho} \right)
=
\begin{cases}
\frac{ t}{1- t}\tan^2\left(\frac{\rho\pi}{2}\right), & t\in(0,1),~ \rho\in[0,1/2], \\[8pt]
\frac{ t}{1- t}\tan^2\left(\frac{(1-\rho)\pi}{2}\right), & t\in(0,1),~\rho\in[1/2,1],\\[8pt]
\infty, & t\in\{0,1\}. 
\end{cases}
$$
%The free divisibility indicator of $\lambda_{t,\rho}$ is not continuous at $ t=0$.  
\end{prop}
\begin{proof}
We assume $t\in(0,1)$; otherwise $\lambda_{t,\rho}$ is a delta measure or a Cauchy distribution whose free divisibility indicator is infinity. 
By computing $\eta$-transforms, we get 
 $$
 (\lambda_{t,\rho})^{\uplus u} = D_{(1-t)u +t}\! \left(\lambda_{\frac{t}{(1-t)u+t}, \rho}\right),\qquad u>0.  
 $$ 
 From Proposition \ref{exacauchy}, this is FID if and only if $|\cos\rho\pi| \leq \frac{2t}{(1-t)u+t}-1$. From (\ref{free div ind char}), $\phi(\lambda_{t,\rho})$ is the solution $u$ of the equation $|\cos\rho\pi| = \frac{2t}{(1-t)u+t}-1$, giving the assertion of the proposition. 
\end{proof}

\subsection{Free Jurek Class and $\mathcal{B}_{\alpha,\rho}$}
The second-named author and Thorbj{\o}rnsen established the free analogue of Yamazato's theorem, saying that any freely selfdecomposable distribution is unimodal \cite{HT}. We want to find examples of freely selfdecomposable distributions from measures in $\mathcal{B}_{\alpha,\rho}$, but such is not possible at least for positive and symmetric cases. This is because scale mixtures of Boolean stable laws are compound free Poisson distributions (see Proposition \ref{free compound beta}), but a nontrivial freely selfdecomposable distribution does not have a finite  L\'evy measure. Instead, we will consider a class called the free Jurek class that is larger than the class of freely selfdecomposable distributions. We consider scale mixtures of Boolean stable laws which belong to the free Jurek class. 

The classical Jurek class was studied in \cite{J85} and it coincides with all the distributions of stochastic integrals of the form $\int_0^1 t\, d X_t$, where $(X_t)_{t\geq0}$ is a L\'evy process starting at 0. 
\begin{defi}
An FID distribution $\mu$ is said to be \textbf{freely s-selfdecomposable} if 
the L\'evy measure $\nu_\mu$ is unimodal with mode 0. The set $\U$ of all freely s-selfdecomposable distributions is called the \textbf{free Jurek class}. 
\end{defi}

We quote a special case of \cite[Theorem 4]{HS}. 
\begin{lem}\label{unimodal boole}
\begin{enumerate}[\rm(1)]
\item The positive Boolean stable law $\mathbf{b}_{\alpha,1}$ is unimodal with mode 0 if and only if $\alpha\in(0,\alpha_0]$, where $\alpha_0=0.7364\dots$ is the unique solution $x\in(0,1)$ of the equation $\sin \pi x =x$. 
\item The symmetric Boolean stable law $\mathbf{b}_{\alpha,1/2}$ is unimodal with mode 0 if and only if $\alpha\in(0,1]$. 
\end{enumerate}
\end{lem}
The following result follows from Khintchine's characterization of unimodality (see \cite{K38} or \cite[Theorem 2.7.3]{Z86}): a probability measure $\mu$ is unimodal with mode 0 if and only if $\mu=\mathbf{u}\circledast \rho$ for some $\rho\in\mathcal{P}$, where $\mathbf{u}$ is the uniform distribution on $[0,1]$. 
\begin{lem}\label{unimodal}
If $\mu \in\mathcal{P}$ is unimodal with mode 0, then so is $\nu\circledast\mu$ for any $\nu\in\mathcal{P}$. 
\end{lem}

\begin{thm}\label{thm909} Let $\alpha_1=\alpha_0/(1+\alpha_0) =0.4241\dots$, where $\alpha_0$ is the number defined in Lemma \ref{unimodal boole}. The following statements hold. 
\begin{enumerate}[\rm(1)]
\item $\mathcal{B}_{\alpha,1} \subset \U$ for $\alpha \in(0,\alpha_1]$. 
\item $\mathcal{B}_{\alpha,1/2} \subset \U$ for $\alpha \in(0,1/2]$. 
\end{enumerate}
\end{thm}
\begin{proof} Let $\mu\in\mathcal{P}_+$, $\beta=\alpha/(1-\alpha)$ and assume that $\rho=1/2$ or $1$.  Note first that $\mu^{1/\alpha}\circledast\mathbf{b}_{\alpha,\rho}$  is FID if $\alpha \leq \alpha_1< 1/2, \rho=1$ or  if $\alpha\leq1/2 <2/3,\rho=1/2$ from Theorem \ref{FIDB}. 
The computation of $S$-transforms gives us the formula
\begin{equation}\label{eq0901}
\mathbf{b}_{\alpha,\rho}=\MP\boxtimes (\mathbf{f}_{1-\alpha,1})^{\boxtimes 1/\beta} \boxtimes  \mathbf{b}_{\beta,\rho}. 
\end{equation}
Then we have
\begin{equation}
\begin{split}
\mu^{1/\alpha}\circledast\mathbf{b}_{\alpha,\rho}
&=\mu^{\boxtimes 1/\alpha}\boxtimes\mathbf{b}_{\alpha,\rho} \\
&=\mu^{\boxtimes 1/\alpha}\boxtimes\MP\boxtimes(\mathbf{f}_{1-\alpha,1})^{\boxtimes 1/\beta}\boxtimes  \mathbf{b}_{\beta,\rho}\\
&=\MP\boxtimes  \left((\mu^{\boxtimes \frac{1}{1-\alpha}}\boxtimes \mathbf{f}_{1-\alpha,1})^{\boxtimes 1/\beta} \boxtimes  \mathbf{b}_{\beta,\rho}\right)\\
&= \MP\boxtimes  \left((\mu^{\boxtimes \frac{1}{1-\alpha}}\boxtimes \mathbf{f}_{1-\alpha,1})^{1/\beta} \circledast  \mathbf{b}_{\beta,\rho}\right), 
\end{split}
\end{equation}
where we used (\ref{c-f1}) on the first and last lines. This means that the L\'evy measure of $\mu^{1/\alpha}\circledast\mathbf{b}_{\alpha,\rho}$ is given by the probability measure 
\begin{equation}\label{levy boole}
(\mu^{\boxtimes \frac{1}{1-\alpha}}\boxtimes \mathbf{f}_{1-\alpha,1})^{1/\beta} \circledast  \mathbf{b}_{\beta,\rho}; 
\end{equation}
see Remark \ref{rem free poisson}. 
If $\alpha \leq \alpha_1$ and $\rho=1$, then $\beta \leq \alpha_0$ and $\mathbf{b}_{\beta,1}$ is unimodal with mode 0 from Lemma \ref{unimodal boole}, and so is the L\'evy measure from Lemma \ref{unimodal}. The symmetric case $\rho=1/2$ is similar. 
\end{proof}

\begin{rem} \begin{enumerate}[\rm(1)]
\item As usual, we cannot use the $S$-transform for general $\rho$, and we cannot extend Theorem \ref{thm909} for $\rho \neq 0,1/2,1$. 

\item The L\'evy measure (\ref{levy boole}) can be written as 
$$
\mathbf{B}_{\alpha/(1-\alpha),\rho}(\mathbf{F}_{1-\alpha,1}(\mu)) 
$$
with notations in Theorem \ref{thm3.3}. 
\end{enumerate}
\end{rem}

\section{Existence of Free Bessel Laws}

In this section, we prove Theorem  \ref{Thm3} which also settles the problem of definition of free Bessel laws stated in Banica et al.\ \cite{BBCC}.

\subsection{Free Powers of Free Poisson}
Given $\mu\in\mathcal{P}_+$, one can ask whether the convolution powers $\mu^{\boxtimes s}$ and $\mu^{\boxplus t}$ exist for various values of $s, t > 0$. Specifically, the question is whether $S_\mu(z)^s$ and $t \mathcal{C}_\mu^\boxplus(z)$ are the $S$- and free cumulant transforms of some probability measures. 
It is known that for $s>1$ or $t>1$, the convolution powers  $\mu^{\boxtimes s}$ and $\mu^{\boxplus t}$ always exist as probability measures.

Furthermore, one can ask whether the convolution powers $(\mu^{\boxtimes s})^{\boxplus t}$ or $(\mu^{\boxplus t})^{\boxtimes s}$ exist, for different values of $s, t > 0$. Since we have the following ``commutation" relation,
$$(\mu^{\boxplus t})^{\boxtimes s} = D_{t^{s-1}}(\mu^{\boxtimes s})^{\boxplus t},$$ for $t\geq 1$ and $s \geq 1$
then both questions are equivalent. 

 We answer this question for the case when $\mu=\MP$. Since the free Poisson distribution $\MP$ is $\boxtimes$-infinitely divisible and free regular (the latter meaning that $\MP^{\boxplus t}\in\mathcal{P}_+$ for any $t>0$; see \cite{AHS}), then the double power $\tilde{\MP}_{st}=(\MP^{\boxtimes s})^{\boxplus t}$ exists as a probability measure when $\max(s,t)\geq1$.

 The moments and cumulants of $\tilde{\MP}_{st}$  were studied by Hinz and M{\l}otkowski \cite{HinMlo}. In particular, they state the following conjecture which is closely related to the question of the possible parameters of free Bessel laws in Banica et al.\ \cite{BBCC} as we explain below.

\begin{conj}[\cite{HinMlo}]\label{conj} $\tilde{\MP}_{st}$ is a probability measure if and only if $\max(s,t)\geq1$. Equivalently, the sequence given by $\tilde{m}_0(s,t)=1$ and
\begin{equation}\label{moments1}
\tilde{m}_n(s,t)=\sum_{k=1}^n\frac{t^k}{n}\binom{n}{k-1}\binom{ns}{n-k}
\end{equation}
is positive definite if and only if $\max(s,t)\geq1$. 
\end{conj}

We solve the conjecture of Hinz and M{\l}otkowski in the affirmative. We will use the relation between free Poisson, free stable and Boolean stable laws
$$
\mathbf{b}_{\alpha,1}=\MP^{\boxtimes \frac{1-\alpha}{\alpha} }\boxtimes \mathbf{f}_{\alpha,1}
$$
proved in (\ref{sym1}) together with the following lemma. 

\begin{lem}\label{lem100}
For $1/2<\alpha<1$ and $t<1$, $(\mathbf{b}_{\alpha,1})^{\boxplus t}$ does not exist as a probability measure.
\end{lem}
\begin{proof}
The Boolean stable law $\mathbf{b}_{\alpha,1}$ is not FID for $\alpha\in(1/2,1)$ \cite{AHb}, and so $\phi(\mathbf{b}_{\alpha,1})<1$ from Theorem \ref{thm001}(\ref{div ii}). From the arguments in Subsection 3.3 of \cite{AH2}, the free divisibility indicator $\phi(\mathbf{b}_{\alpha,1})$ is either 0 or $\infty$, but since it is smaller than 1, it is 0. Then Theorem \ref{thm001}(\ref{div i}) implies the conclusion. 
\end{proof}

Now, we are in position to prove the first part of Theorem \ref{Thm3}. That is, we prove the conjecture by Hinz and M{\l}otkowski.

\begin{proof}[Proof of Conjecture \ref{conj}]
Let $0<s,t<1$, $\alpha=\frac{1}{1+s}$ and $K=t^{1-1/\alpha}$. Suppose that $\tilde{\MP}_{st}$ is a probability measure. Then so is $b(s,t):=\tilde{\MP}_{st}\boxtimes \mathbf{f}_{\alpha,1}$. Since $t<1$, we can take the $1/t$ free additive power, yielding
\begin{eqnarray*}b(s,t)^{\boxplus 1/t}&=&(\tilde{\MP}_{st}\boxtimes \mathbf{f}_{\alpha,1})^{\boxplus 1/t}\\&=&D_{t}\left((\MP^{\boxtimes s})^{\boxplus t}\right)^{\boxplus 1/t} \boxtimes (\mathbf{f}_{\alpha,1})^{\boxplus 1/t}\\
&=&D_{t}(\MP^{\boxtimes s})\boxtimes D_{t^{-1/\alpha}}(\mathbf{f}_{\alpha,1})\\
&=&D_{K}(\MP^{\boxtimes s}\boxtimes  \mathbf{f}_{\alpha,1}), 
\end{eqnarray*}
where we used (\ref{dilation}) in the second equality and the stability property of $\mathbf{f}_{\alpha,1}$ in the third. Since $\alpha=\frac{1}{s+1}$, then $s=(1-\alpha)/\alpha$ and $s\in(0,1)$. Thus, we have proved that, for $1/2<\alpha<1$,
$$
b(s,t)^{\boxplus 1/t}=D_{K}(\MP^{\boxtimes \frac{1-\alpha}{\alpha} }\boxtimes \mathbf{f}_{\alpha,1})=D_{K}(\mathbf{b}_{\alpha,1}).
$$
This means $(\mathbf{b}_{\alpha,1})^{\boxplus t}$ exists as a probability measure but this is impossible by Lemma \ref{lem100}.
\end{proof}

\subsection{Free Bessel Laws}
Let us recall the definition of free Bessel laws, together with some basic facts. The free Bessel laws were introduced in Banica et al.\  \cite{BBCC} as a two-parameter family of probability measures on $\real_+$ generalizing the free Poisson $\MP$. They studied connections with random matrices, quantum groups and $k$-divisible non-crossing partitions. 

The original definition of the free Bessel law is the following.

\begin{defi}
The free Bessel law is the probability measure $\MP_{st}$ with $(s,t)\in(0,\infty)\times(0,\infty)-(0,1)\times(1,\infty)$, defined as follows:
\begin{enumerate}[\rm(1)]
\item For $s\geq1$ we set $\MP_{st}=\MP^{\boxtimes(s-1)}\boxtimes\MP^{\boxplus t}$; 

\item For $t\leq1$ we set $\MP_{st}=((1-t)\delta_0+t\delta_1)\boxtimes\MP^{\boxtimes s}.$
\end{enumerate}
\end{defi}
The compatibility between (1) and (2) comes from the following identity valid for $s\geq1$ and $0<t\leq1$: 
$$
\MP^{\boxtimes (s-1)}\boxtimes\MP^{\boxplus t}=((1-t)\delta_0+t\delta_1)\boxtimes\MP^{\boxtimes s}. 
$$
Special important cases are $t=1$ for which 
$\MP_{s1}=\MP^{\boxtimes s}$ and $s=1$ for which $\MP_{1t}=\MP^{\boxplus t}$.

The moments of free Bessel law $\MP_{st}$ are calculated as follows \cite{BBCC}: 
\begin{equation}\label{fussnarayana} 
m_n(s,t)=\sum_{k=1}^n\frac{t^k}{k}\binom{n-1}{k-1}\binom{ns}{k-1}. 
\end{equation}
In the particular case where $t=1$ and $s$ is an integer we obtain the Fuss-Catalan numbers $m_n=\frac{1}{sn+1}\binom{sn+n}{n},$  known to appear in several contexts. In particular, they count the number of $s$-divisible and $(s+1)$-equal non-crossing partitions. For details on $s$-divisible non-crossing partitions, see  Edelman \cite{Ed}, Stanley \cite{Stan}, Arizmendi \cite{Ar} and Armstrong \cite{Arm}.

Banica et al.\ \cite{BBCC} considered the question of whether $\MP_{st}$ exists as a probability measure for certain points in the critical rectangle $(0, 1)\times (1,\infty)$. The precise range of the parameters $(s, t)$ was an open problem. We can determine it from Conjecture \ref{conj}. Indeed, one recognizes the moments in (\ref{fussnarayana}) as the moments of  $\tilde{\MP}_{s,1/t}$ multiplied by $t^{n+1}$, and  
thus we get the following (or see the proof of \cite[Theorem 3.1]{BBCC}). 
\begin{lem} Let $s,t>0$. We have 
\begin{equation}
\MP_{st}=(1-t)\delta_0+tD_t(\tilde{\MP}_{s,1/t}),
\end{equation}
or equivalently
\begin{equation}\label{eq156}
\tilde{\MP}_{st}=(1-t)\delta_0+tD_t(\MP_{s,1/t}), 
\end{equation}
where equalities are in the sense of linear functionals on the polynomial ring $\comp[x]$, e.g.\ $\MP_{st}(x^n)=m_n(s,t)$, $n\geq0$. 
\end{lem}\label{final}
From the previous lemma we directly get the following, proving the last part of Theorem \ref{Thm3}. 
\begin{cor} $\MP_{st}$ is not a probability measure for $t>1$ and $s<1$.
\end{cor}
\begin{proof}
Suppose $\MP_{s,1/t}$ is a probability measure for some $0<s,t<1$. Then from  (\ref{eq156}), $\tilde{\MP}_{st}$ is a probability measure too. This is a contradiction to Conjecture \ref{conj} which we proved to be true. 
\end{proof}

\end{document}